\documentclass[12pt,a4paper,reqno]{amsart}
\usepackage{amssymb}
\usepackage{mathrsfs}
\usepackage{amsmath}
\usepackage{amsfonts}
\usepackage{graphicx}
\usepackage{bm}
\usepackage[a4paper,left=2cm,right=2cm]{geometry}
\usepackage[colorlinks=true, linkcolor=blue]{hyperref}
\usepackage{caption}
\usepackage{subcaption}
\let\oldmathbf\mathbf
\renewcommand{\mathbf}[1]{\boldsymbol{\oldmathbf{#1}}}
\newtheorem{theorem}{Theorem}
\newtheorem*{theorem*}{Theorem}

\newtheorem{lemma}[theorem]{Lemma}

\newtheorem{proposition}[theorem]{Proposition}
\theoremstyle{definition}
\newtheorem{remark}[theorem]{Remark}
\newtheorem{convention}[theorem]{Convention}

\usepackage{orcidlink}
\title[On a discrete approach to lower bounds in discrepancy theory]{On a discrete approach to lower \\ bounds in discrepancy theory}

\author[L. Brandolini]{Luca Brandolini \orcidlink{0000-0002-9670-9051}}
\address{Dipartimento di Ingegneria Gestionale, dell'Informazione e
della Produzione, Universit{\`a} degli Studi di Bergamo, Viale G. Marconi 5, 24044,
Dalmine BG, Italy}
\email{luca.brandolini@unibg.it}
\email{biancamaria.gariboldi@unibg.it}
\email{giacomo.gigante@unibg.it}
\email{alessandro.monguzzi@unibg.it}
\author[B. Gariboldi]{Bianca Gariboldi \orcidlink{0000-0001-8714-4135}}

\author[G. Gigante]{Giacomo Gigante \orcidlink{0000-0002-1642-679X}}
\author[A. Monguzzi]{Alessandro Monguzzi \orcidlink{0000-0003-3233-5000}} 
\thanks{MSC 2020: 11K38; 39A12; 42A38}
\thanks{This work was supported by the European Union – NextGenerationEU, under the National Recovery and Resilience Plan (NRRP), Mission 4, Component 2, Investment 1.1, funding call PRIN 2022 D.D. 104 published on 2.2.2022 by the Italian Ministry of University and Research (Ministero dell’Università e della Ricerca), Project Title:  TIGRECO - TIme-varying signals on Graphs: REal and COmplex methods – CUP F53D23002630001. The authors are also supported by the Indam--Gnampa project CUP\textunderscore E53C23001670001.}
\keywords{Discrepancy, Lower bounds, Irregularities of distribution, Discrete Fourier transform}

\allowdisplaybreaks
\begin{document}
\begin{abstract}
In this paper, we prove that some renowned lower bounds in discrepancy theory admit a discrete analogue. Namely, we prove that the lower bound of the discrepancy for corners in the unit cube due to Roth holds true also for a suitable finite family of corners. We also prove two analogous results for the  discrepancy on the torus with respect to squares and balls. 
\end{abstract}
\maketitle

\section{Introduction}
The study of irregularities of distribution of point sequences consists in quantifying how far any such sequence necessarily is from being too evenly distributed. There exist several different ways to measure irregularities of distribution, here we will
focus on the discrepancy. In a few words, given a finite sequence of points $\mathcal P_N=\{p_n\}_{n=1}^N$ in a probability space $(X,\mu)$ and a collection $\mathcal A$ of measurable subsets of $X$, the discrepancy function maps any set $A\in\mathcal A$ to the difference between the actual and the expected number of points of the sequence $\mathcal P_N$ that belong to $A$, that is
\[
\mathscr D(A,\mathcal P_N)=\sum_{n=1}^N\chi_A (p_n)-N\mu(A).
\]
One then estimates some norm of this function, like $\sup_{A\in\mathcal A}|\mathcal D(A,\mathcal P_N)|$ or some other $L^p$ norm with respect to certain parameters describing the collection $\mathcal A$. 

For example, in the most classical case one considers $X=\left[  0,1\right)  ^{d}$ with the Lebesgue measure and $\mathcal A=\{A(x)=[  0,x_{1})  \times\cdots\times[
0,x_{d}):x=(x_1,\ldots,x_d)\in (0,1]^d\},
$ the set of anchored boxes (or ``corners''),
so that
\begin{equation*}
\mathscr{D}(  A(x),\mathcal{P}_{N})  
=\sum_{n=1}^{N}\chi_{A(  x)
}(  p_{n})  -Nx_{1}\cdots x_{d}.
\end{equation*}
Roth (\cite{Roth}) proved the existence of a constant $c>0$ independent of $N$ and  $\mathcal{P}_N$ such that
\begin{equation}
\left(\int_{(  0,1]  ^{d}}\left\vert \mathscr{D}(  A(x),\mathcal{P}_{N})
\right\vert ^{2}dx\right)^{\frac{1}{2}}\geqslant c\log^{\frac{d-1}{2}}(  N)  .\label{Roth}%
\end{equation}
In particular, this immediately implies that
\begin{equation}\label{Roth-L-infty}
\sup_{x\in (0,1]^d}|\mathscr{D}(A(x),\mathcal P_N)|\geqslant c\log^{\frac{d-1}{2}}(N).
\end{equation}
Estimate \eqref{Roth} was later extended by Schmidt in \cite{Schmidt77} to any $L^p$ norm, with $p\in(1,\infty)$. Moreover, such estimate is known to be sharp; see \cite[Section 4.1]{Bilyk2} and the references therein. The endpoint situation $p=\infty$, on the contrary, is not completely understood.  In dimension $d=2$, Schmidt \cite{Schmidt72} improved \eqref{Roth-L-infty} proving the sharp estimate
\begin{equation*}
\sup_{x\in (0,1]^2}|\mathscr{D}(A(x),\mathcal P_N)|\geqslant c\log N.
\end{equation*}
A deep and profound problem is how one such sharp estimate should look like in dimension $d\geqslant 3$, but it is beyond the scope of this work to discuss such a problem. Thus, we only mention here the papers \cite{BL, BLV} and refer the reader to \cite[Section 2.1.1]{Bilyk}  for a detailed discussion on this matter and its connection with other conjectures in analysis.

 The literature is rich of excellent sources which provide an overview of the development of geometric discrepancy in the wake of Roth's result and ideas. Among others, we point out \cite{BC,M, CT, Bilyk, Bilyk2}.

A different analysis consists in replacing the set $\mathcal A$ of anchored boxes with the set of all  cubes, or with the set of balls with fixed radius $r<1/2$. More precisely, we shall consider here the $\mathbb Z^d$-periodization of the cubes (or the balls), in such a way that the points of the sequence near the boundary of the unit cube give the same contribution to the discrepancy function as those near the center. In this case the term toroidal discrepancy is often used, as this indeed corresponds to replacing the unit cube $[0,1)^d$ with the $d$-dimensional torus, thus setting $X=\mathbb T^d=\mathbb R^d/\mathbb Z^d$. 

Let us consider cubes first. In this case, the collection $\mathcal A$ is then
\[
\mathcal A=\{x+Q(s):x\in\mathbb T^d,\,s\in(0,1/2)\},
\]
where $Q(s)=[-s,s)^d+\mathbb Z^d$ is the cube of center $0$ and edge $2s$ in the torus, so that 
\[
\mathscr{D}(x+Q(s),\mathcal{P}_N)=
\sum_{n=1}^N\chi_{x+Q(s)}(p_n)-N(2s)^d.
\] 
A classical result due to Hal\'asz \cite{Hal} says that there exists a constant $c$ such that for all integers $N$ and for all sequences $\mathcal P_N=\{p_n\}_{n=1}^N$ in $\mathbb T^d$, 
\[
\left(\int_{0}^{\frac12} \int_{\mathbb T^d} |\mathscr{D}(x+Q(s),\mathcal{P}_N)|^2 dx ds \right)^\frac12\geqslant c\log^{\frac{d-1}{2}}(  N).
\]
See also \cite[Chapter 6]{mont}, \cite{BC} or \cite{M}. It has been proved in \cite{Drmota1996} that the $L^2$ norm of the discrepancy for anchored boxes and of the discrepancy for cubes are asymptotically equivalent, as $N\to+\infty$ (see also \cite{Ruzsa92} for the equivalence of $L^\infty$ norms in the $2$-dimensional case). It then follows that the above considerations on the discrepancy with respect to anchored boxes, such as the sharpness of the $L^2$ estimate, apply to the discrepancy with respect to periodized cubes too.

In the case of balls, the collection $\mathcal A$ is 
\[
\mathcal A=\{x+B_r:x\in\mathbb T^d\}
\]
where $B_r=\{x\in\mathbb R^d:|x|<r\}+\mathbb Z^d$ is the ball of radius $r$ and center $0$ in the torus, for some $r\in(0,1/2)$, and
\[
\mathscr{D}(x+B_r,\mathcal{P}_N)=\sum_{n=1}^N\chi_{x+B_r}(p_n)-N|B_r|
,\]
where $|B_r|$ is the Lebesgue measure of the ball $B_r$.
Most known lower estimates on the norm of the discrepancy function in this case involve some type of averaging with respect to the radius (see e.g. \cite{Beck}, but one should also recall Beck's result \cite{Beck2} concerning the discrepancy with respect to the intersection of disks of fixed radius with the unit square). Here we focus on the following result due to Montgomery (see \cite{mon_2radii} or \cite[Chapter 6]{mont}): if $d\not \equiv 1\operatorname{mod}4$, then there exists a constant $c$ such that for every $r\in(0,1/4)$, for every integer $N$ and for every finite sequence $\mathcal P_N=\{p_n\}_{n=1}^N$,
\[
\left(\int_{\mathbb{T}^d}\left\vert \mathscr D( x+B_r,\mathcal P_N) \right\vert ^{2}dx +\int
_{\mathbb{T}^d}\left\vert \mathscr D( x+B_{2r},\mathcal P_N) \right\vert ^{2}dx\right)^\frac12 \geqslant
cN^{\frac12-\frac{1}{2d}}.
\]
 In fact, the estimate was originally proven by Montgomery in the case $d=2$, but the proof can be easily adapted to the case $d\not \equiv 1\operatorname{mod}4$. See also \cite{BGG2022} for the corresponding version on compact homogeneous spaces, and \cite{BMS,BGGM} for the estimates of the discrepancy with respect to balls with a single fixed radius in the case of the sphere and compact homogeneous spaces. 

Now that we have a few examples in mind, let us focus on the discrepancy 
\[
\mathscr D(A,\mathcal P_N)=\sum_{n=1}^N\chi_A(p_n)-N\mu(A)
\]
as a function of the parameters describing the collection $\mathcal A$. Observe that the first term
$\sum_{n=1}^N\chi_A(p_n)$ is a function that takes integer values between $0$ and $N$, while in all our examples the second term $N\mu(A)$ is analytic in the parameters. Figure \ref{figure} shows the functions
\begin{align*}
(x_1,x_2)&\mapsto \sum_{n=1}^N\chi_{[0,x_1)\times[0,x_2)}(p_n)\qquad \text{ anchored boxes, }d=2\\
(x_1,x_2)&\mapsto \sum_{n=1}^N\chi_{(x_1,x_2)+B_r}(p_n)\qquad \text{ balls, } r=1/5,\,d=2\\
(x_1,s)&\mapsto \sum_{n=1}^N\chi_{x_1+[-s,s)}(p_n)\qquad \text{ cubes, }d=1\\
(x_1,x_2)&\mapsto \sum_{n=1}^N\chi_{(x_1,x_2)+[-s,s)^2}(p_n)\qquad \text{ cubes,  } s=1/5\text{ fixed, }d=2,
\end{align*}
and it highlights the discontinuities of the discrepancy function for $N$ randomly chosen points in the $d$ dimensional unit cube or torus.
\begin{figure}
     \centering
     \begin{subfigure}[b]{0.45\textwidth}
         \centering
         \includegraphics[width=\textwidth]{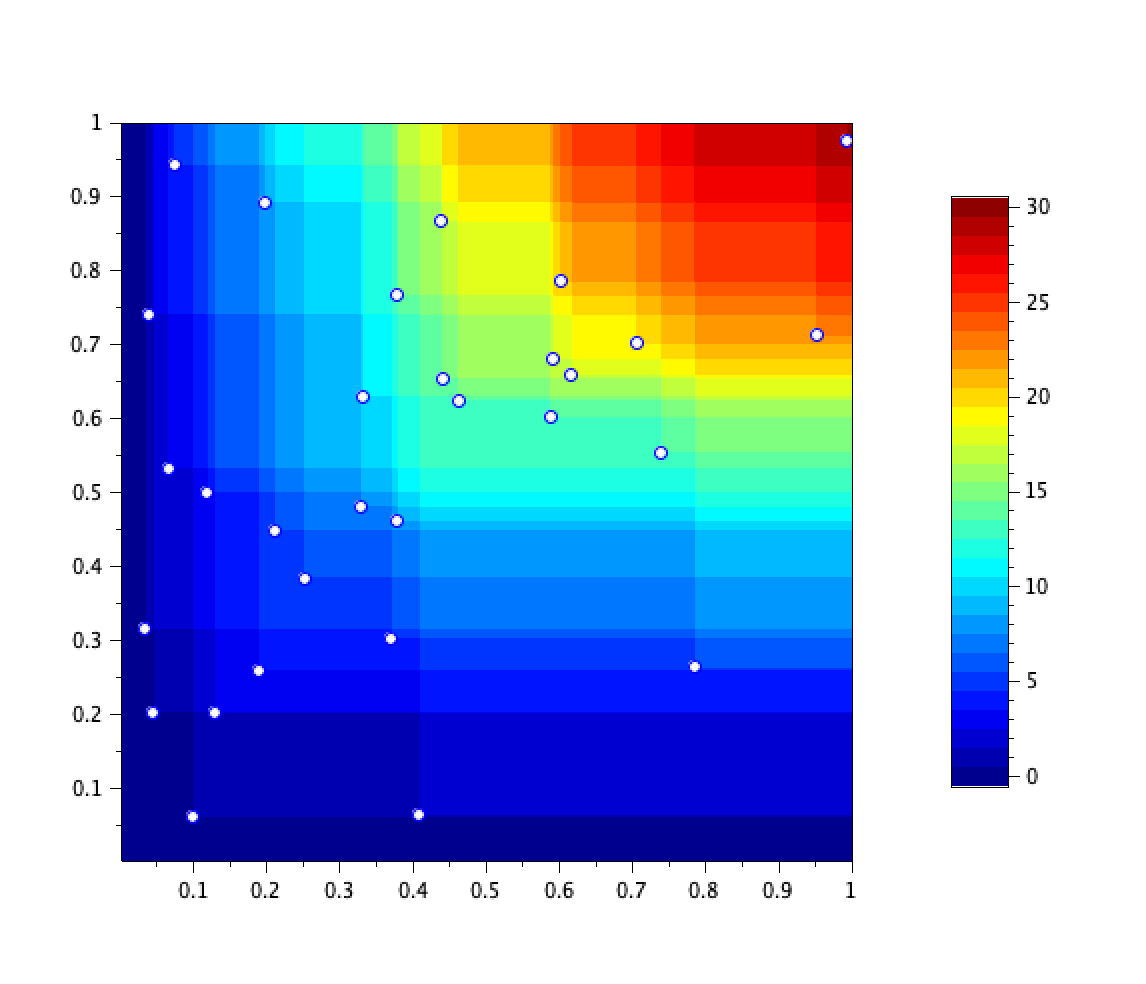}
         \caption{Anchored boxes, $d=2$, $N=30$.}
         \label{boxes}
     \end{subfigure}
     \hfill
     \begin{subfigure}[b]{0.45\textwidth}
         \centering
         \includegraphics[width=\textwidth]{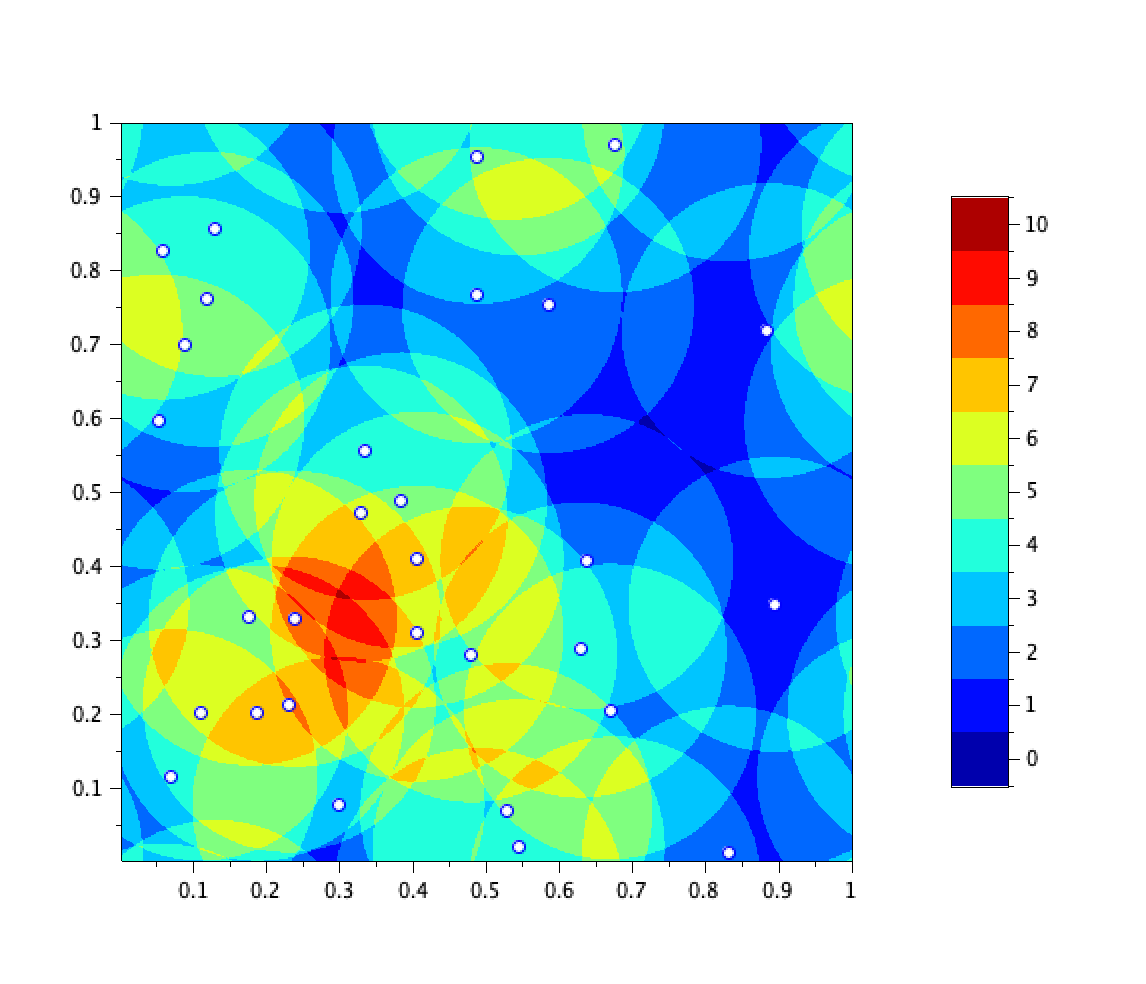}
         \caption{Balls of radius $1/5$, $d=2$, $N=30$.}
         \label{balls}
     \end{subfigure}
     \begin{subfigure}[b]{0.45\textwidth}
         \centering
         \includegraphics[width=\textwidth]{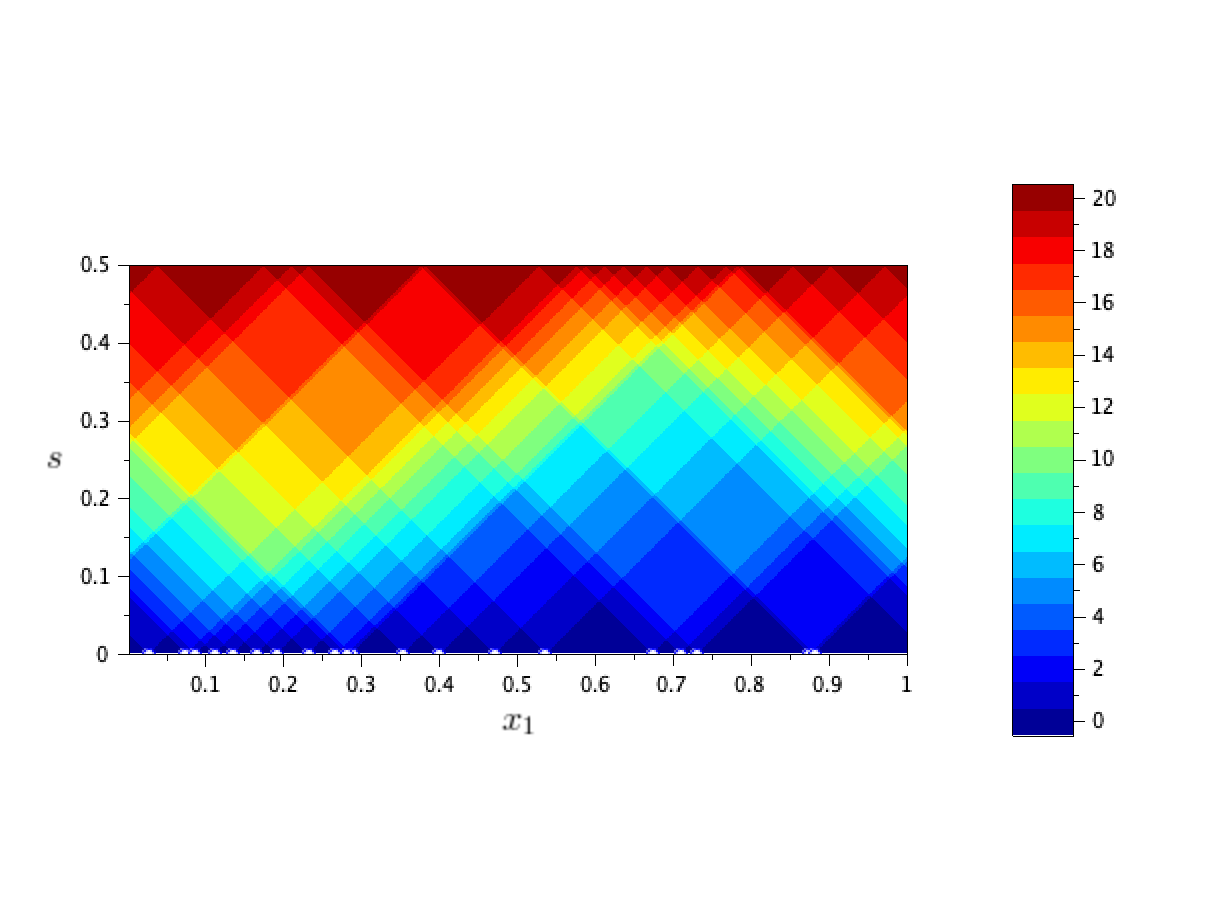}
         \caption{Cubes, $d=1$, $N=20$.}
         \label{segmemnts}
     \end{subfigure}
     \hfill
     \begin{subfigure}[b]{0.45\textwidth}
         \centering
         \includegraphics[width=\textwidth]{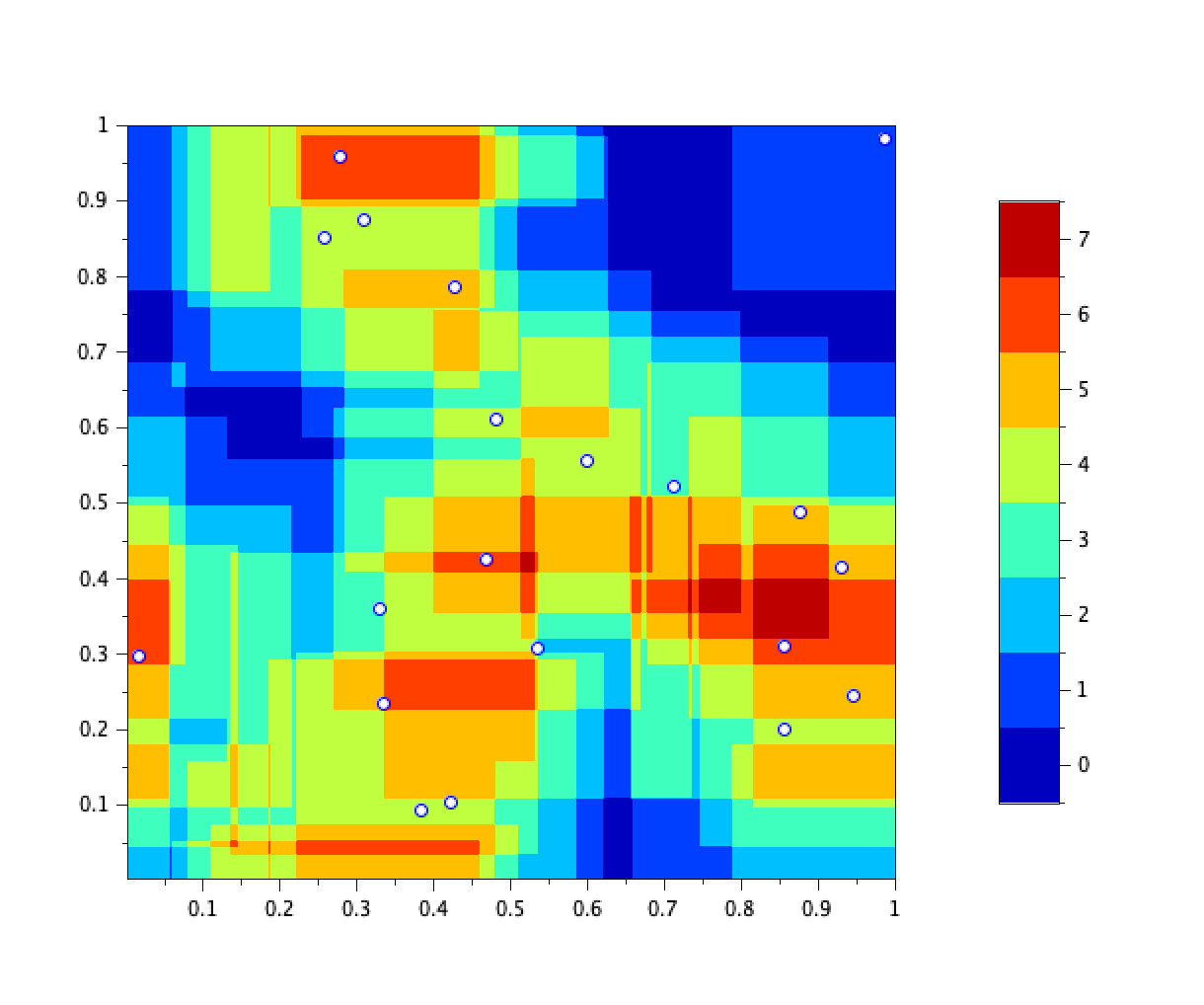}
         \caption{Cubes of fixed side, $d=2$, $N=20$.}
         \label{cubes}
     \end{subfigure}
        \caption{ The counting term of the discrepancy function.}
        \label{figure}
\end{figure}

When the discrepancy function is Riemann integrable, as in all the above examples, one can approximate its $L^2$ norm with the corresponding Riemann sums associated to, say an $M\times\ldots\times M$ grid, for some integer $M$.
For example, in the case of anchored boxes, for any given point sequence $\mathcal P_N$, there exists an integer $M=M(\mathcal P_N)$ depending on $\mathcal P_N$ such that 
\begin{align*}
&\left|\left(\frac{1}{M^{d}}\sum_{j_{1},j_{2},\ldots,j_{d}=1}^{M}\left\vert \mathscr{D}\left(
A\Big(\frac{j_{1}}{M},\cdots,\frac{j_{d}}{M}\Big),\mathcal{P}_{N}\right)  \right\vert
^{2}\right)^{\frac{1}{2}}
-\left(\int_{[0,1)^d}|\mathscr D(A(x),\mathcal P_N)|^2dx\right)^{\frac 12}
\right|\\
&\leqslant \frac 12 \left(\int_{[0,1)^d}|\mathscr D(A(x),\mathcal P_N)|^2dx\right)^{\frac 12}
\end{align*}
and by \eqref{Roth} it follows that
\begin{equation}
\label{discrete}
\left(\frac{1}{M^{d}}\sum_{j_{1},j_{2},\ldots,j_{d}=1}^{M}\left\vert \mathscr{D}\left(
A\Big(\frac{j_{1}}{M},\cdots,\frac{j_{d}}{M}\Big),\mathcal{P}_{N}\right)  \right\vert
^{2}\right)^{\frac{1}{2}}\geqslant \frac c2 \log^{\frac{d-1}2}N.
\end{equation}
Thus, the same lower estimate of the discrepancy holds when the set of all anchored boxes is replaced with the finite set of anchored boxes with upper-right vertex in the grid, and the $L^2$ norm is replaced with an $\ell^2$ norm. 

Unfortunately, though, according to the above argument, the new collection of subsets depends on the sequence $\mathcal P_N$ rather than just on its cardinality $N$. 
Also, due to the discontinuity of the discrepancy function, the above argument does not give any estimate on $M(\mathcal P_N)$. In fact it is not even clear whether the supremum of $M(\mathcal P_N)$ as $\mathcal P_N$ varies in the family of all $N$-point sequences is finite.

Similar observations can be made for the cases of balls or cubes.

In this paper we show lower bounds of the type \eqref{discrete} in the three cases of anchored boxes, cubes and balls described above, with a parameter $M=M_N$ that depends only on the cardinality $N$ of the finite sequence $\mathcal P_N$, thus obtaining a result on irregularity of $N$-point distributions with respect to a given finite collection of sets. We also give upper bounds on the smallest $M_N$ for which the results hold. 

Our results are the following.

\begin{theorem}\label{thm: roth-discreto}
Let $b$ be a positive integer greater than $1$. Let $N$ be a positive integer, let $\nu$ be the integer such that $b^{\nu-2}\leqslant
N<b^{\nu-1}$, and let $M=b^{\nu+\tau}$ for an integer $\tau\ge 1$. Then  there exists an explicit constant $c>0$, depending only on $b$ and $d$, such that, for every finite sequence $\mathcal{P}_{N}=\left\{  p_{n}\right\}  _{n=1}^{N}$ in $\left[  0,1\right)  ^{d}$, 
\[
\left(\frac{1}{M^{d}}\sum_{j_{1},j_{2},\ldots,j_{d}=1}^{M}\left\vert \mathscr{D}\left(
A\Big(\frac{j_{1}}{M},\cdots,\frac{j_{d}}{M}\Big),\mathcal{P}_{N}\right)  \right\vert
^{2}\right)^{\frac{1}{2}}\geqslant c\log^{\frac{d-1}{2}}(  N)  .
\]
In particular, there exists an anchored box $A(j_1/M, \cdots, j_d/M)$ such that 
\begin{equation}\label{discrete-L-infty}
\left|\mathscr{D}\left(A\Big(\frac{j_1}{M},\cdots,\frac{j_d}{M}\Big), \mathcal P_N\right)\right|\geqslant c \log^{\frac{d-1}{2}}(N).
\end{equation}
When $d=2$ there exists an anchored box $A(j_1/M, j_2/M)$ such that the stronger estimate
\[
\left|\mathscr{D}\left(A\Big(\frac{j_1}{M}, \frac{j_2}{M}\Big), \mathcal P_N\right)\right|\geqslant\widetilde{c} \log(N)
\]
holds for an explicit constant $\widetilde{c}$ depending only on $b$. 
\end{theorem}
Set
\[
\mathbb J_M= \Big\{ j\in\mathbb Z : -\frac{M}{2}\leqslant j< \frac{M}{2}\Big\}
\]
and
\[
\mathbb S_M=\Big\{\frac1M, \frac2M,\ldots, \frac12-\frac1M\Big\}.
\]
\begin{theorem}\label{teo_square}
There exists a constant $c>0$ such that, for every positive integer $N$, for every even integer $M\geqslant18dN$ and for every finite sequence $\mathcal{P}_N=\{p_n\}_{n=1}^N$ in $\mathbb{T}^d$, 
   \begin{align*}
   &\left(\frac{1}{(\operatorname{card} \mathbb J_M)^d}\frac{1}{\operatorname{card}\mathbb S_M}\sum_{j_1,\ldots,j_d\in\mathbb J_M}\sum_{s\in\mathbb S_M}\bigg| \mathscr{D}\bigg(\Big(\frac{j_1}{M},\cdots, \frac{j_d}{M}\Big)+Q(s),\mathcal{P}_N\bigg)\bigg|^2\right)^{\frac{1}{2}}\geqslant c\log^{\frac{d-1}{2}}(  N). 
   \end{align*}
\end{theorem}

\begin{theorem}\label{teo_balls}
Let $d\not \equiv 1\operatorname{mod}4$. Then there exist constants
$C,c>0$ such that, for every $0<r<1/4$,  for every integer $N$, for
every even integer $M$ with
\begin{equation}
M\geqslant C {N^{1+\frac{1}{  2d} }}r^{-1},\label{M da sotto}%
\end{equation}
and for every finite sequence $\mathcal{P}_{N}=\left\{  p_{n}\right\}_{n=1}^N$ in  $\mathbb{T}^d$, we
have%
\begin{align*}
\bigg(\frac{1}{(\operatorname{card}\mathbb{J}_M)^{d}}&\sum_{k=1,2}\,\,\sum_{j_1,\ldots,j_d\in\mathbb{J}_M}  \left\vert \mathscr{D} \bigg(\Big(\frac{j_1}{M},\ldots, \frac{j_d}{M}\Big)+B_{kr},\mathcal{P}_N\bigg)
\right\vert ^{2}\bigg)^{\frac{1}{2}}
\geqslant cr^\frac{d}{2}N^{\frac{1}{2}-\frac{1}{2d}}.
\end{align*}
\end{theorem}

One should first observe that the classical theorems of Roth and Schmidt for boxes, of Hal\'asz for cubes, and Montgomery for balls follow immediately from the above results, by simply letting $M\to+\infty$.

Concerning the irregularity of $N$-point distributions with respect to finite collections of subsets, one should observe that it may easily happen that there exist sequences with vanishing discrepancy function. For example, if $N=K^d$ for some positive integer $K$, then the finite sequence $\mathcal P_N=\{(j_1/K,\ldots,j_d/K):j_1,\ldots,j_d=0,\ldots,K-1\}$ has zero discrepancy with respect to all anchored boxes $A(j_1/K,\ldots,j_d/K)$, for $j_1,\ldots,j_d=1,\ldots, K$. This shows that in \eqref{discrete} one needs $M>N^{1/d}$. A second example, involving a collection of boxes with a finer resolution, is the following: if $b$ is a power of a prime number, $b\geqslant d-1$, then by \cite[p. 125]{D-P} there exist point sets of $N=b^{\nu-2}$ points with zero discrepancy with
respect to all $b$-adic boxes with volume $N^{-1}$. See Remark \ref{rmk-roth-discreto}. 

As we already mentioned, in our research it is important that the nodes and weights of the Riemann sums used to approximate the $L^2$ norms be previously identified and independent of the specific $N$-point sequence, as the final estimate can then be interpreted as a result on irregularity of $N$-point distributions with respect to a given finite collection of sets. It should be mentioned here that, if one is willing to accept that the nodes and weights of the Riemann sum depend on the specific $N$-point sequence, then a simple argument can be given, at least in the case of anchored boxes. Let $\mathcal L$ be the grid generated by the finite sequence $\mathcal P_N$, meaning that $x\in\mathcal L$ if and only if every coordinate of $x$ coincides with the corresponding coordinate of a point of $\mathcal P_N$. $\mathcal L$ naturally partitions $(0,1]^d$ into disjoint boxes of the form $C=\Pi_{i=1}^d(a_i,b_i]$, and we call $\mathcal C$ the collection of all such boxes
(see Figure \ref{boxes} to get an idea).
One can easily show that for every $C\in\mathcal C$ and for every $x\in C$, if one calls $x(C,0)$ and $x(C,1)$ respectively the lower left and upper right vertices of $C$,
\[
|\mathscr D(A(x),\mathcal P_N)|\leqslant \max\{
|\mathscr D(\overline{A(x(C,0))},\mathcal P_N)|,\,|\mathscr D({A(x(C,1))},\mathcal P_N)|\}.
\]
Therefore
\[
\int_{(0,1]^d}|\mathscr D(A(x),\mathcal P_N)|^2dx
\leqslant \sum_{C\in\mathcal C}|C|\max\left\{
|\mathscr D(\overline{A(x(C,0))},\mathcal P_N)|^2,\,|\mathscr D({A(x(C,1))},\mathcal P_N)|^2
\right\}.\]
It follows that Roth's lower estimate \eqref{Roth} implies the discrete $\ell^2$ estimate
\[
\left(\sum_{C\in\mathcal C}|C|\max\left\{
|\mathscr D(\overline{A(x(C,0))},\mathcal P_N)|^2,\,|\mathscr D({A(x(C,1))},\mathcal P_N)|^2
\right\}\right)^{\frac 12}\geqslant c\log^{\frac{d-1}2}(N).
\]
Notice that the cardinality of $\mathcal C$ is at most $(N+1)^d$. 
Similar observations could be made in the case of cubes or balls, but it should be noticed that the family $\mathcal C$ in this case would be much more difficult to identify (again, see Figure \ref{figure}).

 Rather than using some sort of approximation argument, our strategy aims at a direct proof of the estimates, following -- when possible -- the lines of the corresponding continuous classical results, and dealing with the nontrivial technical complications that arise from working in a discrete setting. In particular, the computations to obtain certain discrete Fourier transforms in the proofs of Theorem \ref{teo_square} and Theorem \ref{teo_balls} are somewhat different from their continuous analogues.
 
\section{Proof of the discrete version of Roth's lower bound}
The proof of Theorem \ref{thm: roth-discreto} follows from a straightforward modification of the Haar basis proof of the classical result (see, e.g., \cite{Bilyk}). To prove the theorem let us first recall some definitions and preliminary facts. A $b$-adic interval is an interval
of the form%
\[
I=\Big(  \frac{a}{b^{r}},\frac{a+1}{b^{r}}\Big]
\]
with $a$ and $r$ positive integers and $0\leqslant a<b^{r}$. A $b$-adic box in $\mathbb{R}^d$ is a box of the form $R=I_1\times\cdots\times I_d$ where $I_1,\ldots,I_d$ are $b$-adic intervals. For any $b$-adic interval $I$ as above define the function
\[
h_{I}(  x)  =
\begin{cases}
-1 & x\in\left(  \frac{a}{b^{r}},\frac{a}{b^{r}}+\frac{1}{b^{r+1}}\right],\\
1 & x\in\left(  \frac{a}{b^{r}}+\frac{1}{b^{r+1}},\frac{a}{b^{r}}+\frac
{2}{b^{r+1}}\right],\\
0 & \text{otherwise}%
\end{cases}
\]
and for any $b$-adic box $R=I_{1}\times\cdots\times I_{d}$ define the function
\[
h_{R}(  x)  =h_{I_{1}}(  x_{1})  \cdots h_{I_{d}}(
x_{d})  .
\]
Let us consider the index set
\[
\mathbb{H}_{\nu}^{d}=\left\{  \mathbf{r}=(r_1,\ldots, r_d)\in\mathbb{N}^{d}:r_{1}+\cdots
+r_{d}=\nu\right\},
\]
where $\mathbb N$ denotes the set of non-negative integers,
and observe that  
\[
\operatorname{card}(  \mathbb{H}_{\nu}^{d})=\binom{\nu+d-1}{d-1}
\geqslant\frac{\nu^{d-1}}{(d-1)!}.
\]
For a fixed $\mathbf{r}\in \mathbb H^d_\nu$, let $\mathcal{D}_{\mathbf{r}}^{d}$ be the collection of
$b$-adic boxes $R=I_{1}\times\cdots\times I_{d}$ such that
\[
I_{j}=\Big(  \frac{a_j}{b^{r_{j}}},\frac{a_j+1}{b^{r_{j}}}\Big]
\]
with $0\leqslant a_j< b^{r_j}$. Observe that every $R\in\mathcal D^d_{\mathbf{r}}$ has volume $b^{-\nu}$. It will be important for the sequel to observe that since $M=b^{\nu+\tau}$ for an integer $\tau\ge1$, if $I=(a/b^r, (a+1)/b^r]$ with $r\leqslant \nu$, then
\begin{equation}\label{h_I-mean}
\frac{1}{M}\sum_{j=1}^{M} h_I\Big(\frac{j}{M}\Big)=0.
\end{equation}
Hence,
\begin{equation*}
\frac{1}{M^d}\sum_{j_1,\ldots, j_d=1}^{M} h_{R}\Big(\frac{j_1}{M},\ldots, \frac{j_d}{M}\Big)=0
\end{equation*}
for every rectangle $R\in\mathcal D^d_{\mathbf{r}}$ where $\mathbf{r}\in\mathbb H^d_{\nu}$.

Finally, define the inner product
\[
\langle f,g\rangle= \frac{1}{M^d}\sum_{j_1,\ldots,j_d=1}^{M}f\Big(\frac{j_1}{M},\cdots, \frac{j_d}{M}\Big)g\Big(\frac{j_1}{M},\cdots, \frac{j_d}{M}\Big)
\]
and denote by $\|\cdot\|_2$ the associated norm. Analogously we can define the norms $\|\cdot\|_1$ and $\|\cdot\|_{\infty}$.

Notice that the collection of functions $\{h_R\}_{R\in\mathcal D^d_{\mathbf r}, \mathbf{r}\in \mathbb H^d_\nu}$ is an orthogonal family with respect to this inner product. 

We are ready to prove Theorem \ref{thm: roth-discreto}.

\begin{proof}[Proof of Theorem \ref{thm: roth-discreto}]
We follow the proof of (\ref{Roth}) given in \cite{Bilyk}. The main difference is that we use a  $b$-adic Haar system, instead of a dyadic one, which allows to take into account the scale based on powers of $b$. We want to identify a suitable test function $F$ and use the estimate
\[
\|\mathscr{D}(A(\cdot), \mathcal P_N)\|_2\geqslant \frac{\langle \mathscr{D}(A(\cdot), \mathcal P_N),F\rangle}{\|F\|_2}.
\]
The test function $F$ will be given by a sum of Haar functions. Let $\mathbf{r}\in\mathbb H^d_\nu$. We first compute the inner product
\[
\langle \mathscr{D}(A(\cdot), \mathcal P_N), h_R\rangle
\]
where $R= I_1\times\cdots\times I_d$ is a $b$-adic box in $\mathcal D^d_{\mathbf{r}}$.

To simplify the notation in the rest of the proof we set $D_N(\cdot)=\mathscr{D}(A(\cdot), \mathcal P_N)$ and we also set $q=(q_1,\ldots, q_d)\in\mathcal P_N$. We have
\begin{align*}
\langle D_N,h_R\rangle &=\frac{1}{M^d}\sum_{j_1,\ldots,j_d=1}^{M} D_{N}\Big(\frac{j_1}{M},\cdots,\frac{j_d}{M}\Big)h_{I_1}\Big(\frac{j_1}{M}\Big)\cdots h_{I_d}\Big(\frac{j_d}{M}\Big)\\
&=\frac{1}{M^d}\sum_{j_1,\cdots,j_d=1}^{M}\sum_{q\in\mathcal P_N}\chi_{\left[0,\frac{j_1}{M}\right)}(q_1)\cdots\chi_{\left[0, \frac{j_d}{M}\right)}(q_d)h_{I_1}\Big(\frac{j_1}{M}\Big)\cdots h_{I_d}\Big(\frac{j_d}{M}\Big) \\
&\qquad -\frac{1}{M^d}\sum_{j_1,\cdots, j_d=1}^{M} N\frac{j_1}{M}\cdots \frac{j_d}{M}h_{I_1}\Big(\frac{j_1}{M}\Big)\cdots h_{I_d}\Big(\frac{j_d}{M}\Big)\\
&= \mathcal{A}-N\mathcal{B}.
\end{align*}
Since $\mathbf{r}\in\mathbb{H}_{\nu}^{d}$, $R\in\mathcal{D}_{\mathbf{r}}^{d}$, if  $R\cap\mathcal{P}_{N}=\emptyset$, then%
\begin{align*}
\mathcal{A}&=\frac{1}{M^{d}}\sum_{j_{1},j_{2},\ldots,j_{d}=1}^{M}\sum_{q\in
\mathcal{P}_{N}}\chi_{\left(  q_{1},1\right]  }\Big(  \frac{j_{1}}{M}\Big)
\cdots\chi_{\left(  q_{d},1\right]  }\Big(  \frac{j_{d}}{M}\Big)
h_{I_1}\Big(\frac{j_1}{M}\Big)\cdots h_{I_d}\Big(\frac{j_d}{M}\Big)  =0.
\end{align*}
Indeed, if $q\in \mathcal P_N$ is not in $R$, then there exists $s\in\{1,\ldots, d\}$ such that $q_s\notin I_s$ and therefore, by \eqref{h_I-mean}, 
\[
\sum_{j_s=1}^{M}\chi_{(q_s,1]}\Big(\frac{j_s}{M}\Big) h_{I_s}\Big(\frac{j_s}{M}\Big)=0.
\]
It remains to compute
\begin{align*}
\mathcal{B}&  =\frac{1}{M^{d}}\sum_{j_{1}=1}^{M}\frac{j_{1}}{M}h_{I_{1}}\Big(
\frac{j_{1}}{M}\Big)  \cdots\sum_{j_{d}=1}^{M}\frac{j_{d}}{M}h_{I_{d}%
}\Big(  \frac{j_{d}}{M}\Big)  .
\end{align*}
Since for $I=(a/b^r, (a+1)/b^r]$ we have
\begin{align*}
\sum_{j=1}^{M}\frac{j}{M}h_{I}\Big(  \frac{j}{M}\Big)&=-\sum_{j=1}^{M}\frac{j}{M}\chi_{\left(  \frac{a}{b^{r}%
},\frac{a}{b^{r}}+\frac{1}{b^{r+1}}\right]  }\Big(  \frac{j}{M}\Big)  +\sum_{j=1}^{M}\frac{j}{M}\chi_{\left(  \frac{a%
}{b^{r}}+\frac{1}{b^{r+1}},\frac{a}{b^{r}}+\frac{2}%
{b^{r+1}}\right]  }\Big(  \frac{j}{M}\Big)  \\
& =-\sum_{j=M\frac{a}{b^{r}}+1}^{M\frac{a}{b^{r}}+\frac
{M}{b^{r+1}}}\frac{j}{M}+\sum_{j=M\frac{a}{b^{r}}%
+\frac{M}{b^{r+1}}+1}^{M\frac{a}{b^{r}}+\frac{2M}{b^{r+1}}%
}\frac{j}{M}\\
&=-\sum_{j=M\frac{a}{b^{r}}+1}^{M\frac{a%
}{b^{r}}+\frac{M}{b^{r+1}}}\frac{j}{M}+\sum_{j=M\frac{a%
}{b^{r}}+1}^{M\frac{a}{b^{r}}+\frac{M}{b^{r+1}}}\frac
{j+\frac{M}{b^{r+1}}}{M}\\
& =\sum_{j=M\frac{a}{b^{r}}+1}^{M\frac{a}{b^{r}}+\frac
{M}{b^{r+1}}}\frac{1}{b^{r+1}}=\frac{1}{b^{r+1}}\Big(  \frac
{M}{b^{r+1}}\Big)  =\frac{M}{b^{2}b^{2r}}=\frac{M}{b^{2}}\left\vert
I\right\vert ^{2},
\end{align*}
we obtain%
\[
\mathcal{B} =\frac{1}{b^{2d}}\left\vert R\right\vert ^{2}=\frac{1}{b^{2d+2\nu}}.
\]
In conclusion, if $\mathbf{r}\in\mathbb{H}_{\nu}^{d}$, $R\in\mathcal{D}_{\mathbf{r}}^{d}$ and $R\cap\mathcal{P}_{N}=\emptyset$, then
\begin{equation}\label{D_N-h_R}
\langle D_N, h_R\rangle = -\frac{N}{b^{2d+2\nu}}.
\end{equation}
Now, for every $\mathbf{r}\in\mathbb{H}_{\nu}^{d}$ define the function
\[
f_{\mathbf{r}}(  x)  =-\sum_{R\in\mathcal{D}_{\mathbf{r}}^{d}%
,R\cap\mathcal{P}_{N}=\emptyset}  h_{R}(  x)  .
\]
Observe that $\|f_{\mathbf{r}}\|_2\leqslant\|f_{\mathbf{r}}\|_\infty\leqslant 1$.
Since there are $b^\nu$ boxes in $\mathcal D_{\mathbf r}^d$ and less than $b^{\nu-1}$ points in $\mathcal P_N$, it follows that \[\operatorname{card}(\{R\in\mathcal{D}_{\mathbf{r}}^{d},R\cap\mathcal{P}%
_{N}=\emptyset\})\geqslant b^\nu-b^{\nu-1}\]
and, therefore,
\begin{align*}
\langle D_N, f_{\mathbf{r}}\rangle &= -\sum_{R\in\mathcal{D}_{\mathbf{r}}^{d},R\cap\mathcal{P}%
_{N}=\emptyset}\langle D_N, h_{R}\rangle
=\sum
_{R\in\mathcal{D}_{\mathbf{r}}^{d},R\cap\mathcal{P}_{N}=\emptyset}\frac{N}{b^{2d+2\nu}}\\
& \geqslant\frac{b^{\nu-2}}{b^{2d+2\nu}}(  b^\nu-b^{\nu-1}) 
 =\frac{b-1}{b^{2d+3}}.
\end{align*}
Finally, set
\[
F(  x)  =\sum_{\mathbf{r}\in\mathbb{H}_{\nu}^{d}}f_{\mathbf{r}%
}(  x)  .
\]
We have, by the orthogonality of the $f_{\mathbf{r}}$'s, that
\begin{align*}
\|F\|^2_2&=\sum_{\mathbf{r}%
\in\mathbb{H}_{\nu}^{d}}\|f_{\mathbf{r}}\|^2_2\leqslant \operatorname{card}(  \mathbb{H}_{\nu}^{d}).
\end{align*}
In conclusion, since%
\begin{align}\label{D_NF}
\langle D_N, F\rangle&= \sum_{\mathbf{r}\in\mathbb H^d_\nu} \langle D_N, f_{\mathbf{r}}\rangle\geqslant \sum_{\mathbf{r}\in\mathbb H^d_\nu}\frac{b-1}{b^{2d+3}}=\operatorname{card}(  \mathbb{H}_{\nu}^{d})\frac{b-1}{b^{2d+3}},
\end{align}
we get
\begin{align*}
\operatorname{card}(  \mathbb{H}_{\nu}^{d})\frac{b-1}{b^{2d+3}}  &\leqslant \langle D_N, F\rangle \leqslant \|D_{N}\|_2 \|F\|_2\leqslant  \big(\operatorname{card}(  \mathbb{H}_{\nu}^{d})\big)^{\frac12}\|D_{N}\|_2.
\end{align*}
Therefore,
\begin{align*}
\| D_N\|^2_2&\geqslant \operatorname{card}(  \mathbb{H}_{\nu}^{d})\left(  \frac{b-1}{b^{2d+3}}\right)  ^{2}\geqslant \frac{\nu^{d-1}}{(d-1)!}\left(  \frac{b-1}{b^{2d+3}}\right)  ^{2} \geqslant \frac{1}{(d-1)!}\left(  \frac{b-1}{b^{2d+3}}\right)  ^{2}\big(  \log
_{b}(  N)  \big)  ^{d-1}.
\end{align*}
This concludes the proof of the first part of the theorem. Let us now set $d=2$. Following \cite[Chapter 6]{M}, we want to exploit the inequality
\begin{align*}
|\langle D_N, G\rangle| &\leqslant \|D_N\|_{\infty}\| G\|_{1}.
\end{align*}
As test function $G$ we take 
\[
G(  x) =\prod_{\mathbf{r}\in\mathbb{H}_{\nu}^{2}}\big(1+\kappa f_{\mathbf{r}%
}(  x)\big)-1 =\prod_{r=0}^{\nu} \big(1+\kappa f_{r%
}(  x)\big)-1 ,
\]
where $\mathbf{r}=(r,\nu-r)$  and we identify $f_r(x)$ with $f_{(r,\nu-r)}(x)$, and $\kappa\in (0,1)$ is a free parameter that will be chosen later. Expanding the above product, we can write 
\begin{align*}
G(x)&=\kappa\sum_{r=0}^\nu f_r(x)+\sum_{\ell\geqslant 2} \kappa^\ell G_\ell(x)= \kappa G_1+\sum_{\ell\geqslant 2} \kappa^\ell G_\ell(x),
\end{align*}
where  
\[
G_\ell(x)=\sum_{0\leqslant r_1< r_2< \ldots< r_\ell\leqslant \nu}f_{r_1}(x)f_{r_2}(x)\cdots f_{r_\ell}(x).
\]
Let us consider a single product $f_{r_1}\cdots f_{r_\ell}$. This is a sum of terms of the form
$
h_{R_{r_1}}\cdots h_{R_{r_\ell}},
$
where each $R_r$ is a rectangle in $\mathcal D^d_{(r,\nu-r)}$ not containing any point of $\mathcal P_N$, say
\[
R_r=\Big(\frac{a_r}{b^r},\frac{a_r+1}{b^{r}}\Big]\times\Big(\frac{a'_r}{b^{\nu-r}},\frac{a'_r+1}{b^{\nu-r}}\Big]= I^{(1)}_r\times I^{(2)}_{\nu-r}.
\]
We have
\[
(h_{R_{r_1}}\cdots h_{R_{r_\ell}})(x_1,x_2)= \prod_{i=1}^{\ell} h_{I^{(1)}_{r_i}}(x_1)\prod_{k=1}^{\ell} h_{I^{(2)}_{\nu-r_k}}(x_2).
\]
Since $ r_1< r_2< \ldots< r_\ell $, this is either $0$ or $\pm h_{I^{(1)}_{r_\ell}}(x_1)h_{I^{(2)}_{\nu-r_1}}(x_2)$, and
by \eqref{h_I-mean} it follows  that \[\sum_{j_1,j_2=1}^{M} G_{\ell}\Big(\frac{j_1}{M},\frac{j_2}{M}\Big)=0.\]
In particular, for future reference, observe that different choices of the family $R_{r_1},\ldots, R_{r_\ell}$ yield disjoint intersections $R_{r_1}\cap\cdots \cap R_{r_\ell}$. It follows that each $I^{(1)}_{r_\ell}\times I^{(2)}_{\nu-r_1}$ comes from at most one choice of the family  $R_{r_1},\cdots, R_{r_\ell}$. Therefore $f_{r_1}\cdots f_{r_\ell}$ is the sum of at most $b^{\nu+r_\ell-r_1}$ terms of the form $\pm h_{I^{(1)}_{r_\ell}}(x_1)h_{I^{(2)}_{\nu-r_1}}(x_2)$.

Now, since $\|f_{\mathbf r}\|_\infty\le1$ and $0<\kappa<1$,
\begin{align*}
\|G\|_{1}  &\leqslant \frac{1}{M^2}\sum_{j_1,j_2=1}^{M} \prod_{\mathbf{r}\in\mathbb{H}_{\nu}^{2}}\Big(1+\kappa f_{\mathbf{r}%
}\Big( \frac{j_1}M,\frac{j_2}M \Big)\Big)+1 
=2+\frac{1}{M^{2}}\sum_{j_{1},j_{2}=1}^{M}\sum_{\ell=1}^{\nu} \kappa^\ell G_{\ell}\Big(  \frac{j_{1}}{M}%
,\frac{j_{2}}{M}\Big)   =2.
\end{align*}
Let us focus on
\begin{align*}
\langle D_{N}, G\rangle&=\sum_{\ell=1}^{\nu}\kappa^\ell \langle D_N, G_\ell\rangle.
\end{align*}
Notice that, since $G_1=F$, we have as in \eqref{D_NF} 
\[
\langle D_N, G_1\rangle \geqslant\operatorname{card}(  \mathbb{H}_{\nu}^{2})\frac{b-1}{b^{7}}.
\]
For the other terms we proceed as follows. We know that on each rectangle $R$ of the $b^{r_\ell} \times b^{\nu-r_1}$ grid, the product
$f_{r_1}f_{r_2}
\cdots f_{r_\ell}$ either is 0 or looks like $\pm h_R$; in the latter case, $R$ contains no
points of $\mathcal{P}_N$  and therefore, as in \eqref{D_N-h_R},
\[
\langle D_N, h_R\rangle = -\frac{N}{b^4}|R|^2.
\]
Then, 
\begin{align*}
\sum_{\ell=2}^\nu\kappa^\ell\big|\langle G_\ell, D_N\rangle\big|&\leqslant \sum_{\ell=2}^{\nu}\kappa^\ell \sum_{0\leqslant r_1<\cdots<r_\ell\leqslant \nu}\big|\langle f_{r_1}\cdots f_{r_\ell}, D_N\rangle\big|\\
&\leqslant \sum_{\ell=2}^{\nu}\kappa^\ell \sum_{0\leqslant r_1<\cdots<r_\ell\leqslant \nu}\sum_{\substack{R \text{ in the }\\ b^{r_\ell}\times b^{\nu-r_1} \text{ grid}}} \big|\langle h_{R}, D_N\rangle\big|\leqslant \sum_{\ell=2}^\nu \kappa^\ell \sum_{0\leqslant r_1<\cdots<r_\ell\leqslant \nu} Nb^{-\nu-r_\ell+r_1-4}\\
&\leqslant \sum_{\ell=2}^\nu \kappa^\ell \sum_{r_1=0}^{\nu-\ell+1} \sum_{t=\ell-1}^{\nu-r_{1}} \sum_{ r_1<r_2<\cdots<r_{\ell-1}<r_1+t } Nb^{-\nu-t-4}\\
&\leqslant \sum_{\ell=2}^\nu \kappa^\ell \sum_{r_1=0}^{\nu} \sum_{t=\ell-1}^{\nu}  Nb^{-\nu-t-4} \binom{t-1}{\ell-2}= N\sum_{r_1=0}^{\nu} \sum_{t=1}^{\nu} b^{-\nu-t-4}\sum_{\ell=2}^{t+1} \kappa^\ell  \binom{t-1}{\ell-2}\\
&= N\kappa^2\sum_{r_1=0}^{\nu} \sum_{t=1}^{\nu} b^{-\nu-t-4}\sum_{\ell=0}^{t-1} \kappa^\ell  \binom{t-1}{\ell}
= Nb^{-\nu-4}\kappa^2\sum_{r_1=0}^{\nu} \sum_{t=1}^{\nu} b^{-t}(1+\kappa)^{t-1}\\
&= (Nb^{-\nu+1})b^{-6}\kappa^2\sum_{r_1=0}^{\nu} \sum_{t=1}^{\nu} \left(\frac{1+\kappa}{b}\right)^{t-1}\leqslant b^{-6}\kappa^2(\nu+1)  \frac{b}{b-1-\kappa}.
\end{align*}
It follows that
\[\|D_N\|_{\infty} \geqslant \frac{|\langle D_N,G\rangle|}{\|G\|_1}\geqslant \kappa(\nu+1)\frac{(b-1)b^{-7}-\kappa b^{-5}/(b-1-\kappa)}{2}\geqslant \widetilde c\log(N)
\]
for every $N\geqslant 2$ if $\kappa$ is sufficiently small. 
\end{proof}

We conclude the section pointing out the following.

\begin{remark}\label{rmk-roth-discreto}
If $b$ is a power of a prime number, $b\geqslant d-1$, then by \cite[p. 125]{D-P} there exist point sets of $b^{\nu-2}$ points with zero discrepancy with
respect to all boxes in
\[
\bigcup\limits_{\mathbf{r\in}\mathbb{H}_{\nu-2}^{d}}
\mathcal{D}_{\mathbf{r}}^{d},
\]
that is, all $b$-adic boxes of volume
$b^{-\nu+2}$. Theorem \ref{thm: roth-discreto} shows that there exists a not necessarily $b$-adic box on the
slightly tighter grid $b^{-\nu-1}\mathbb{Z}^{d}$ with at least logarithmic
discrepancy.  
\end{remark}

\section{Proof of the lower bound for squares}

To simplify the notation for the computations, we set $D_N(x,s)=\mathscr{D}(  x+Q(s), \mathcal P_N) $ and $
j= (j_1,\ldots,j_d)$,
so that our aim is to estimate the quantity
\[
\|D_N\|^2_2=\frac{1}{(\operatorname{card} \mathbb J_M)^d}\frac{1}{\operatorname{card}\mathbb S_M}\sum_{j\in \mathbb J^d_M}\sum_{s\in\mathbb S_M}| D_N(M^{-1}j,s)|^2.
\]
(Notice that the same symbol $\|\cdot\|_{2}$ denoted a different $L^2$ norm in the previous section.
)

Observe that if $j\in\mathbb{J}_M^d$ and $s\in\mathbb{S}_M$, then $D_N(M^{-1}j,s)$ does not change if we substitute every point $p_n=(p_{n,1},\ldots,p_{n,d})$ with the point $\widetilde{p_n}=(\lfloor Mp_{n,1}\rfloor/M,\ldots,\lfloor Mp_{n,d}\rfloor/M  )$. Therefore we can assume that $p_n=M^{-1}z_n$ with $z_n\in \mathbb{J}_M^d$.

Let $F:M^{-1}\mathbb{J}_M^d\to \mathbb{C}$ and define the discrete Fourier transform of $F$ by
\[
\widehat{F}(k)=\sum_{m \in \mathbb J_M^d}F(M^{-1}m)  \exp(  -2\pi ik %
\cdot M^{-1}m).
\]
Observe that $\widehat{F}$ is periodic, that is for every $h\in\mathbb{Z}^d$ we have $\widehat{F}(k+Mh)=\widehat{F}(k)$.  Furthemore, we have the inversion formula,
\[F(M^{-1}m)=\frac{1}{M^d}\sum_{k \in \mathbb J_M^d}\widehat{F}(k)  \exp(  2\pi ik %
\cdot M^{-1}m)\]
and the Plancherel identity (see \cite[Chapter 7]{SS})
\begin{equation}\label{Plancherel}\sum_{m\in\mathbb{J}_M^d}\vert
F( M^{-1} m)  \vert ^{2}=\frac{1}{M^{d}}\sum_{k
\in\mathbb{J}_M^d}\vert \widehat{F}(
k)  \vert ^{2}.\end{equation}
In the next two lemmas we compute the discrete Fourier transform of the discrepancy in order to use it to compute $\|D_N\|_2$.
\begin{lemma}\label{FT}
For every $s\in\mathbb S_M$, the discrete Fourier transform of the function $D_N(\cdot,s)$ is given by
\[\widehat{D}_N(  k,s)=\begin{cases}
\displaystyle\sum_{n=1}^{N}\exp(
-2\pi ik\cdot M^{-1}z_n)  \widehat{\chi}_{-Q(s)  }(  k) & \text{if } k\in\mathbb{J}_M^d\setminus\{0\},\\[3mm]
0 & \text{if } k=0.
\end{cases}\]  
\end{lemma}
\begin{proof}
By the definition of discrete Fourier transform,
\begin{align*}
\widehat{D}_N(  k,s)   &  =\sum_{m \in \mathbb J_M^d}D_N(M^{-1}m,s)  \exp(  -2\pi ik %
\cdot M^{-1}m) \\
&=\sum_{m \in \mathbb J_M^d}\left(\sum_{n=1}^{N}\chi_{M^{-1}m+Q(s)  }(  p_{n})  -(2s)  ^{d}N \right)  \exp(  -2\pi ik %
\cdot M^{-1}m)\\
&  =\sum_{m \in \mathbb J_M^d}\sum_{n=1}^{N}\chi_{M^{-1}m+Q(
s)  }(  M^{-1}z_n)  \exp(  -2\pi ik\cdot
M^{-1}m) \\
&\quad  -\sum_{m \in \mathbb J_M^d}(  2s)  ^{d}N\exp(  -2\pi
ik\cdot M^{-1}m) \\
&  =\sum_{n=1}^{N}\sum_{m \in \mathbb J_M^d}\chi_{Q(  s)  }\big(
M^{-1}(z_{n}-m)\big)  \exp(  -2\pi ik\cdot
M^{-1}m)  -(  2s)  ^{d}NM^{d}\delta_0(
k)  .
\end{align*}

Then,
\begin{align*}
&  \sum_{m \in \mathbb J_M^d}\chi_{Q(  s)  }\big(
M^{-1}(z_{n}-m)\big)  \exp(  -2\pi ik\cdot
M^{-1}m)   \\
  = & \sum_{m \in \mathbb J_M^d}\chi_{-Q(s)  }(
M^{-1}m)  \exp(  -2\pi ik\cdot M^{-1}(
m+z_n)) \\
 = &  \exp(  -2\pi ik\cdot M^{-1}z_n)  \sum_{m \in \mathbb J_M^d}\chi_{-Q(s)  }(  M^{-1}m%
)  \exp(  -2\pi ik\cdot M^{-1}m) \\
  = & \exp(  -2\pi ik\cdot M^{-1}z_n)  \widehat{\chi
}_{-Q(s)  }(k)
\end{align*}
and, therefore,
\[
\widehat{D}_N(  k,s)  =\sum_{n=1}^{N}\exp(
-2\pi ik\cdot M^{-1}z_n)  \widehat{\chi}_{-Q(s)  }(  k)  -(  2s)  ^{d}NM^{d}%
\delta_0(k)  .
\]
Moreover,
\[
\widehat{\chi}_{-Q(s)  }(0)  =\sum_{m \in \mathbb J_M^d}\chi_{-Q(s)  }(  M^{-1}m%
)  =(  2s)  ^{d}M^{d}.
\]
\end{proof}
\begin{lemma}\label{conto_importante}
\[
\|D_N\|^2_2=\frac{1}{M^{2d}(M/2-1)}\sum_{0\neq k
\in\mathbb{J}_M^d}\left|  \sum_{n=1}%
^{N}\exp(  -2\pi i k\cdot M^{-1}z_{n})  \right|
^{2}\sum_{s\in\mathbb{S}_M}\left|  \widehat{\chi}_{-Q(s)  }(
k)  \right|  ^{2}.
\]
\end{lemma}
\begin{proof} It follows from the Plancherel identity \eqref{Plancherel} and Lemma \ref{FT} that
\begin{align*}
&\sum_{s\in\mathbb{S}_M}\sum_{j\in\mathbb{J}_M^d}\vert
D_N(  M^{-1}j,s)  \vert ^{2}
 = \frac{1}{M^{d}}\sum_{s\in\mathbb{S}_M}\sum_{k
\in\mathbb{J}_M^d}\vert \widehat{D}_N(
k,s)  \vert ^{2}\nonumber\\
&  = \frac{1}{M^{d}}\sum_{0\neq k
\in\mathbb{J}_M^d}\left|  \sum_{n=1}%
^{N}\exp(  -2\pi i k\cdot M^{-1}z_{n})  \right|
^{2}\sum_{s\in\mathbb{S}_M}\left|  \widehat{\chi}_{-Q(s)  }(
k)  \right|  ^{2}.
\end{align*}
\end{proof}
To proceed with the proof we now focus on the quantity 
\begin{equation}\label{discrete-plancherel}
\sum_{s\in\mathbb{S}_M}\left|  \widehat{\chi}_{-Q(s)  }(
k)  \right|^2=\sum_{r=1}^{M/2-1}\left|  \widehat{\chi}_{-Q(
r/M)  }(k)  \right|  ^{2}.
\end{equation}
We will estimate from below such quantity in Proposition \ref{prop7}, but in order to do it we need a series of preliminary lemmas in which we prove a more explicit formulation of \eqref{discrete-plancherel}.
\medskip

\begin{convention} \label{conv:seno}From now on we assume that for every integer $r>0$ the
function
\[
\frac{\sin( \pi rx) }{\sin( \pi x)  }%
\]
is extended by continuity at the integers and therefore it takes the value
\[
r  \frac{\cos(\pi r x)}{\cos(\pi x)}=
r(-1)^{(r-1)x}%
\]
when $x\in\mathbb{Z}$.
\end{convention}

\medskip

\begin{lemma}
Let $r=1,\ldots,M/2-1$ and let $k=(k_1,\ldots,k_d)\in\mathbb{J}_M^d$. Then
\[
\left|  \widehat{\chi}_{-Q\left(  r/M\right)  }(k)
\right|  =\prod_{u=1}^{d}\left|  \frac{\sin(  2\pi k_{u}r/M)
}{\sin(  \pi k_{u}/M)  }\right|  .
\]

\end{lemma}

\begin{proof}
We have
\begin{align*}
\widehat{\chi}_{-Q\left(  r/M\right)  }(k)   &  =
\sum_{m \in \mathbb J_M^d}\chi_{-Q\left(  r/M\right)  }(  M^{-1}%
m)  \exp(  -2\pi ik\cdot M^{-1}m) \\
&  = \prod_{u=1}^{d}\sum_{m_{u}=-r+1}^{r}\exp(  -2\pi ik_{u}%
m_{u}/M)  =\prod_{u=1}^{d}\Phi_{r}(  k_{u})  ,
\end{align*}
where \[\Phi_{r}(  k_u)  =\sum_{m=-r+1}^{r}\exp(  -2\pi
ik_um/M)  .\] 
Let now $k_u\neq0$. Then
\begin{align*}
\Phi_{r}(k_u)   &  = \sum_{m=-r+1}^{r}\exp(  -2\pi ik_u/M)
^{m}=\sum_{m=0}^{2r-1}\exp(  -2\pi ik_u/M)  ^{m-r+1}\\
&  = \exp(  -2\pi ik_u/M)  ^{-r+1}\sum_{m=0}^{2r-1}\exp(  -2\pi
ik_u/M)  ^{m}\\
&  = \exp(  -2\pi ik_u/M)  ^{-r+1}\frac{1-\exp(  -2\pi
ik_u/M)  ^{2r}}{1-\exp(  -2\pi ik_u/M)  }\\
&  = \frac{\exp(  -2\pi ik_u/M)  ^{-r+1}-\exp(  -2\pi
ik_u/M)  ^{r+1}}{1-\exp(  -2\pi ik_u/M)  }\\
&  = \frac{\exp(  -2\pi ik_u/M)}{\exp(-\pi ik_u/M)} \cdot \frac{\exp(  -2\pi ik_u/M)
^{-r}-\exp(  -2\pi ik_u/M)  ^{r}}{\exp(  \pi ik_u/M)
-\exp(  -\pi ik_u/M)  }\\
&  = \exp(  -\pi ik_u/M)  \frac{\sin(  2\pi k_ur/M)  }%
{\sin(  \pi k_u/M)  }.
\end{align*}
Observe that, by Convention \ref{conv:seno}, when $k_u=0$ we obtain the correct value $2r$. Hence
\[
\left|  \widehat{\chi}_{-Q\left(  r/M\right)  }(k)
\right|  =\prod_{u=1}^{d}\left|  \frac{\sin(  2\pi k_{u}j/M)
}{\sin(  \pi k_{u}/M)  }\right|
\]
for every $k\in\mathbb{J}_M^d$.
\end{proof}

\begin{lemma}
\label{lem:Somma Trasformata}Let $k\in\mathbb{J}_M^d$, and assume that
$k_{1}\cdots k_{h}\neq0$ and $k_{h+1},\ldots, k_{d}=0$, for some $h\geqslant 1$. Then
\begin{equation}\label{eq:somma_trasformata}
\sum_{r=1}^{M/2-1}\left|  \widehat{\chi}_{-Q(  r/M)  }(
k)  \right|  ^{2}=\frac{4^{d}}{2^{3h}\prod_{u=1}^{h}\sin
^{2}(  \pi k_{u}/M)  }\sum_{r=1}^{M/2-1}r^{2d-2h}\prod_{u=1}%
^{h}\big(  1-\cos(  4\pi k_{u}r/M)  \big).
\end{equation}

\end{lemma}

\begin{proof}
We have
\begin{align*}
\sum_{r=1}^{M/2-1}&\left\vert \widehat{\chi}_{-Q\left(  r/M\right)  }(k)  \right\vert ^{2}  =\sum_{r=1}^{M/2-1}\prod_{u=1}%
^{h}\left\vert \frac{\sin(  2\pi k_{u}r/M)  }{\sin(  \pi
k_{u}/M)  }\right\vert ^{2}(  2r)  ^{2d-2h}\\
&  =\frac{1}{\prod_{u=1}^{h}\left\vert \sin(  \pi k_{u}/M)
\right\vert ^{2}}\sum_{r=1}^{M/2-1}\left(  2r\right)  ^{2d-2h}\prod_{u=1}%
^{h}\left\vert \sin(  2\pi k_{u}r/M)  \right\vert ^{2}\\
&  =\frac{4^{d}}{2^{3h}\prod_{u=1}^{h}\left\vert \sin(  \pi
k_{u}/M)  \right\vert ^{2}}\sum_{r=1}^{M/2-1}r^{2d-2h}\prod_{u=1}%
^{h}\big(  1-\cos(  4\pi k_{u}r/M)  \big) .
\end{align*}

\end{proof}
The next lemma is essential in order to use Lemma \ref{lem:Somma Trasformata} to estimate \eqref{discrete-plancherel} from below.

\begin{lemma}
\label{groups}Let $0<\varepsilon<1/2$ and let $k$ be an integer such that  $0<\left\vert k\right\vert
\leqslant\varepsilon M$. Then%
\[
\operatorname{card}\left\{  r=0,\ldots,M/2-1:1-\cos(  4\pi kr/M)
\leqslant1-\cos(  2\pi\varepsilon)  \right\}  \leqslant2\varepsilon M .
\]
\end{lemma}

\begin{proof}
Let $D=\gcd(  k,M/2)  $. Let us show that the set%
\[
\left\{  \exp\Big(2\pi i\frac{k}{M/2}r\Big)\right\}  _{r=0}^{M/2-1}%
\]
contains $\frac{M/2}{D}$ elements and that every value of the list repeats $D$ times. Let
$z=\exp(2\pi i\frac{1}{M/2})$. It is well known that $z$ generates a cyclic
group of order $M/2$ and that $z^{k}$ generates a cyclic group of $\frac
{M/2}{D}$ elements. Let us denote this cyclic group by $G$. The elements of
this group are evenly distributed in the unit circle and separated by the angle
 $
\frac{4\pi D}{M}$. For a fixed $\varepsilon>0$
we have%
\[
\operatorname{card}\{  g\in G:\left\vert \arg(  g)
\right\vert \leqslant2\pi\varepsilon\}  \leqslant\frac{4\pi\varepsilon
}{\frac{4\pi D}{M}}+1=\frac{\varepsilon M}{D}+1.
\]
Since every element $\exp(2\pi i\frac{k}{M/2}r)$ repeats $D$ times, it follows that%
\begin{align*}
&  \operatorname{card}\left\{  r=0,\ldots, M/2-1:\left\vert \arg\bigg(  \exp\Big(2\pi
i\frac{k}{M/2}r\Big)\bigg)  \right\vert \leqslant2\pi\varepsilon\right\}    \leqslant D\left(  \frac{\varepsilon M}{D}+1\right)  =\varepsilon M+D.
\end{align*}
Since $\left\vert k\right\vert \leqslant\varepsilon M$, we have $D\leqslant
\varepsilon M$ and therefore%
\[
\operatorname{card}\left\{  r=0,\ldots, M/2-1:\left\vert \arg\bigg(  \exp\Big(2\pi
i\frac{k}{M/2}r\Big)\bigg)  \right\vert \leqslant2\pi\varepsilon\right\}
\leqslant2\varepsilon M.
\]
\end{proof}

Using the above lemma, we provide the following estimate for the sum in the right-hand side of \eqref{eq:somma_trasformata}.

\begin{proposition}\label{prop6}
Let $0<\varepsilon<\frac{1}{8h}$. Let $k=(k_1,\ldots, k_h)\in\mathbb J_M^h\cap [-\varepsilon M,\varepsilon M]^h$ be such that $k_1\cdots k_h\neq 0$.
Then%
\[
\sum_{r=\lfloor M/4\rfloor}^{M/2-1}\prod_{u=1}^{h}\big(  1-\cos(  4\pi k_{u}r/M)
\big)  \geqslant\frac{M}{4}\eta_{h}(  \varepsilon)
\]
with%
\[
\eta_{h}(  \varepsilon)  =\left(  1-8\varepsilon h\right)  \big(
1-\cos(  2\pi\varepsilon)  \big)  ^{h}.
\]

\end{proposition}

\begin{proof}
For every $u=1,\ldots,h$, let
\[
W_{u}=\left\{  r=\lfloor M/4\rfloor,\ldots,M/2-1:1-\cos(  4\pi k_{u}r/M)
>1-\cos(  2\pi\varepsilon)  \right\}
\]
and let
$
W=\bigcap\limits_{u=1}^{h}
W_{u}$.
Since $W^{c}=
\bigcup\limits_{u=1}^{h}
W_{u}^{c}$, by Lemma \ref{groups} we have 
\[
\operatorname{card}W^{c}\leqslant 2\varepsilon Mh 
\]
and therefore%
\[
\operatorname{card}W\geqslant\frac{M}{2}-\left\lfloor\frac{M}{4}\right\rfloor-2\varepsilon Mh\geqslant\frac{M}%
{4}(  1-8\varepsilon h).
\]
Hence%
\begin{align*}
\sum_{r=\lfloor M/4\rfloor}^{M/2-1}\prod_{u=1}^{h}\bigg(  1-\cos\Big(  2\pi\frac{k_{u}%
r}{M/2}\Big)  \bigg)   &  \geqslant\sum_{r\in W}\prod_{u=1}^{h}\bigg(
1-\cos\Big(  2\pi\frac{k_{u}r}{M/2}\Big)  \bigg)  \\
&  \geqslant\sum_{r\in W}\big(  1-\cos(  2\pi\varepsilon)
\big)  ^{h}\geqslant\frac{M}{4}(  1-8\varepsilon h)  \big(
1-\cos(  2\pi\varepsilon) \big)  ^{h}.
\end{align*}

\end{proof}

Finally, we have all the ingredients to estimate \eqref{discrete-plancherel}.

\begin{proposition}\label{prop7}
Let $0<\varepsilon<\frac{1}{8d}$. 
Let $k\neq 0$ such that $k=(k_1,\ldots,k_d)\in \mathbb J_M^d\cap [-\varepsilon M,\varepsilon M]^d$.
Then
 \[\sum_{r=1}^{M/2-1}\left\vert \widehat{\chi}_{-Q\left(  r/M\right)  }(k) \right\vert ^{2}\geqslant 2\frac{4^d\eta_{d}(
\varepsilon)  }{\pi^{2d}\big(  \prod_{u=1}^{d}%
\max(  1, |k_{u}|)  \big)  ^{2}}\left\lfloor \frac{M}{4}\right\rfloor^{2d+1},\]
with 
\[
\eta_{d}(  \varepsilon)  =\left(  1-8\varepsilon d\right)  \big(
1-\cos(  2\pi\varepsilon)  \big)  ^{d}.
\]
\end{proposition}

\begin{proof} Let $h\geqslant 1$ the number of  nonzero components of $k$. Without loss of generality, we can assume that these are $k_1,\ldots,k_h$. We may therefore apply Lemma \ref{lem:Somma Trasformata} and Proposition \ref{prop6},
\begin{align*}
\sum_{r=1}^{M/2-1}\left\vert \widehat{\chi}_{-Q\left(  r/M\right)  }(k)  \right\vert ^{2}&=\frac{4^{d}}{2^{3h}\prod_{u=1}^{h}%
\sin^{2}(  \pi k_{u}/M)  }\sum_{r=1}^{M/2-1}r^{2d-2h}\prod
_{u=1}^{h}\big(  1-\cos(  4\pi k_{u}r/M)  \big)  \\
& \geqslant\frac{4^{d}}{2^{3h}\prod_{u=1}^{h}%
\sin^{2}(  \pi k_{u}/M)  }\left\lfloor \frac{M}{4}\right\rfloor^{2d-2h}\sum_{r=\lfloor M/4\rfloor}^{M/2-1}\prod
_{u=1}^{h}\big(  1-\cos(  4\pi k_{u}r/M)  \big)  \\
& \geqslant\frac{4^{d}}{2^{3h}\prod_{u=1}^{h}%
\sin^{2}\left(  \pi k_{u}/M\right)  }\left\lfloor \frac{M}{4}\right\rfloor^{2d-2h} \frac{M}{4}\eta_{h}(  \varepsilon)\\
& \geqslant\frac{4^{d}}{2^{-h}\prod_{u=1}^{h}%
\big(M\sin(  \pi k_{u}/M)  \big)^2}\left\lfloor \frac{M}{4}\right\rfloor^{2d+1} \eta_{h}(  \varepsilon).
\end{align*}
Since%
\[
M|\sin(  \pi k_{u}/M) | \leqslant\pi |k_{u}|,
\]
we obtain%
\begin{align*}
\sum_{r=1}^{M/2-1}\left\vert \widehat{\chi}_{-Q\left(  r/M\right)  }(k) \right\vert ^{2}&\geqslant \frac{4^{d}\eta_{h}(  \varepsilon)}{2^{-h}\big(\prod_{u=1}^{h}%
\pi k_{u}\big)^2}\left\lfloor \frac{M}{4}\right\rfloor^{2d+1} \\ &\geqslant2\frac{4^d\eta_{d}(
\varepsilon)  }{\pi^{2d}\big(  \prod_{u=1}^{d}%
\max(  1, |k_{u}|)  \big)  ^{2}}\left\lfloor \frac{M}{4}\right\rfloor^{2d+1}.
\end{align*}
\end{proof}
To be finally able to prove Theorem \ref{teo_square} we need another step. As in the classical proof in \cite{mont}, we need to refine further the estimate in the above proposition. In particular, we want to estimate the term $(\prod_{u=1}^{d}\max(1,k_u))^2$, and this will be done thanks to a pair of intermediate results we now state.

\begin{lemma}\label{log}
For every $t\in\mathbb{R}$ and every positive integer $\ell$, we have%
\[t^{2}\geqslant \left(  \frac{2e}{\ell}\right)
^{\ell}
(  \log_{+}t)  ^{\ell}.
\]
\end{lemma}
\begin{proof}
This follows from the fact that the function $e^{x}/x^\ell$ has its minimum at $x=\ell$. 
\end{proof}

Let $p$ be a positive integer to be chosen later and for $x_1,\ldots,x_{d-1}>0$ let
\[
R_{\left(  x_{1},\ldots,x_{d-1}\right)  }=[-x_{1},x_{1}]\times\cdots
\times\left[  -x_{d-1},x_{d-1}\right]  \times\left[  -\frac{p}{x_{1}\cdots
x_{d-1}},\frac{p}{x_{1}\cdots x_{d-1}}\right]  .
\]

\begin{proposition}\label{int}
For every $k\in\mathbb{J}_M^d$,
\begin{align*}
&  \int_{1}^{p}\int_1^{\frac{p}{x_1}}\cdots\int_{1}%
^{\frac{p}{x_{1}\cdots x_{d-2}}}\chi_{R_{(  x_{1},\ldots,x_{d-1})
}}(k)\frac{dx_{d-1}}{x_{d-1}}\cdots\dfrac{dx_2}{x_2}\dfrac{dx_{1}}{x_{1}}\\
&  \qquad\qquad =\frac{1}{(  d-1)  !}\Big(  \log_{+}\frac{p}{\max(
1,\left\vert k_{1}\right\vert )  \cdots\max(  1,\left\vert
k_{d}\right\vert )  }\Big)  ^{d-1}.
\end{align*}
\end{proposition}

\begin{proof}
We have
\begin{align*}
&  \int_{1}^{p}\int_1^{\frac{p}{x_1}}\cdots \int_{1}%
^{\frac{p}{x_{1}\cdots x_{d-2}}}\chi_{R_{(  x_{1},\ldots,x_{d-1})
}}(k)\frac{dx_{d-1}}{x_{d-1}}\cdots\dfrac{dx_2}{x_2}\dfrac{dx_{1}}{x_{1}}\\
&  =\int_{1}^{p}\chi_{\left[  -x_{1},x_{1}\right]  }(k_{1})\int_{1}^{\frac
{p}{x_{1}}}\chi_{\left[  -x_{2},x_{2}\right]  }(k_{2})\cdots\\
&\quad  \times \int_{1}^{\frac
{p}{x_{1}\cdots x_{d-2}}}\chi_{\left[  -x_{d-1},x_{d-1}\right]  }(k_{d-1}%
)\chi_{\big[  -\frac{p}{x_{1}\cdots x_{d-1}},\frac{p}{x_{1}\cdots x_{d-1}%
}\big]  }(k_{d})\frac{dx_{d-1}}{x_{d-1}}\cdots\dfrac{dx_{1}}{x_{1}}\\
&  =\int_{1}^{p}\chi_{[|k_1|,+\infty)
}(x_{1})\int_{1}^{\frac{p}{x_{1}}}\chi_{[|k_2|,+\infty) }(x_{2})\cdots\\
&\quad \times \int_{1}^{\frac{p}{x_{1}\cdots
x_{d-2}}}\chi_{[|k_{d-1}|,+\infty)
}(x_{d-1})\chi_{\big[1,\frac{p}{x_{1}\cdots x_{d-2}\left\vert
k_{d}\right\vert }\big]  }(x_{d-1})\frac{dx_{d-1}}{x_{d-1}}\cdots
\dfrac{dx_{1}}{x_{1}}.%
\end{align*}
When $k_d=0$, we assume that $\big[1,\frac{p}{x_{1}\cdots x_{d-2}\left\vert
k_{d}\right\vert }\big]=[1,+\infty)$. 
We start evaluating the innermost integral. Observe that this can be nonzero only if the three intervals $[|k_{d-1}|,+\infty)$, $\big[1,\frac{p}{x_{1}\cdots x_{d-2}\left\vert
k_{d}\right\vert }\big]$ and $\big[1,\frac{p}{x_{1}\cdots x_{d-2} }\big]$ have nonempty intersection, that is if \[\frac{p}{x_{1}\cdots
x_{d-2}\max(  1,\left\vert k_{d}\right\vert )  }\geqslant
\max(  1,\left\vert k_{d-1}\right\vert )  .\] 
In this case
\begin{align*}
& \int_{1}^{\frac{p}{x_{1}\cdots x_{d-2}}}\chi_{[\left\vert
k_{d-1}\right\vert,+\infty)  }(x_{d-1})\chi_{\big[1,\frac{p}{x_{1}\cdots x_{d-2}\left\vert
k_{d}\right\vert }\big]  }(x_{d-1}%
)\frac{dx_{d-1}}{x_{d-1}}\\
& =\int_{\max\left(  1,\left\vert k_{d-1}\right\vert \right)  }^{\frac
{p}{x_{1}\cdots x_{d-2}\max\left(  1,\left\vert k_{d}\right\vert \right)  }%
}\frac{dx_{d-1}}{x_{d-1}}=\log\frac{p}{x_{1}\cdots x_{d-2}\max(
1,\left\vert k_{d-1}\right\vert )  \max(  1,\left\vert
k_{d}\right\vert )  }.
\end{align*}
Overall, since for  $\frac{p}{x_{1}\cdots x_{d-2}\max(  1,\left\vert
k_{d}\right\vert )  }<\max(  1,\left\vert k_{d-1}\right\vert
)  $ the integral is zero, we obtain
\begin{align*}
& \int_{1}^{\frac{p}{x_{1}\cdots x_{d-2}}}\chi_{[\left\vert
k_{d-1}\right\vert,+\infty)  }(x_{d-1})\chi_{\big[1,\frac{p}{x_{1}\cdots x_{d-2}\left\vert
k_{d}\right\vert }\big]  }(x_{d-1}%
)\frac{dx_{d-1}}{x_{d-1}}\\
& =\log_{+}\frac{p}{x_{1}\cdots x_{d-2}\max(  1,\left\vert k_{d-1}%
\right\vert )  \max(  1,\left\vert k_{d}\right\vert )  }.
\end{align*}
The desired result now follows iterating the following identity. For any $n=1,\ldots,d-2$, 
\begin{align}
& \int_{1}^{\frac{p}{x_{1}\cdots x_{n-1}}}\chi_{[\vert
k_{n}\vert,+\infty)}(x_{n})\Big(  \log_{+}\frac{p}{x_{1}\cdots
x_{n}\max(  1,\left\vert k_{n+1}\right\vert )  \cdots\max(
1,\left\vert k_{d}\right\vert )  }\Big)  ^{d-n-1}\frac{dx_{n}}{x_{n}%
}\nonumber\\
& =\frac{1}{d-n}\Big(  \log_{+}\frac{p}{x_{1}\cdots x_{n-1}\max(
1,\left\vert k_{n}\right\vert )  \cdots\max(  1,\left\vert
k_{d}\right\vert )  }\Big)  ^{d-n}.\label{log^d-n}%
\end{align}
(In our notation when $n=1$ we understand that $p/(x_1\cdots x_{n-1})=p$).

To prove \eqref{log^d-n}, observe that if we let
\[
A=\frac{p}{x_{1}\cdots x_{n-1}\max\left(  1,\left\vert k_{n+1}\right\vert
\right)  \cdots\max\left(  1,\left\vert k_{d}\right\vert \right)  },
\]
then $A\leqslant\frac{p}{x_{1}\cdots x_{n-1}}$. If $\max(  1,\left\vert
k_{n}\right\vert )  <A$, we have
\begin{align*}
& \int_{1}^{\frac{p}{x_{1}\cdots x_{n-1}}}\chi_{[  \left\vert
k_{n}\right\vert ,+\infty)  }(x_{n})\Big(  \log_{+}\frac{p}{x_{1}\cdots
x_{n}\max(  1,\left\vert k_{n+1}\right\vert )  \cdots\max(
1,\left\vert k_{d}\right\vert )  }\Big)  ^{d-n-1}\frac{dx_{n}}{x_{n}%
}\\
& =\int_{\max\left(  1,\left\vert k_{n}\right\vert \right)  }^{A}\Big(
\log_{+}\frac{A}{x_{n}}\Big)  ^{d-n-1}\frac{dx_{n}}{x_{n}}\\
& =\int_{1}^{\frac{A}{\max\left(  1,\left\vert k_{n}\right\vert \right)  }%
}\left(  \log u\right)  ^{d-n-1}\frac{du}{u}=\frac{1}{d-n}\Big(  \log\frac
{A}{\max(  1,\left\vert k_{n}\right\vert )  }\Big)  ^{d-n}\\
& =\frac{1}{d-n}\Big(  \log\frac{p}{x_{1}\cdots x_{n-1}\max(
1,\left\vert k_{n}\right\vert )  \cdots\max(  1,\left\vert
k_{d}\right\vert )  }\Big)  ^{d-n}.
\end{align*}
On the other hand, if $\max(  1,\left\vert k_{n}\right\vert )  >A$,
the integral is zero. Since in this case%
\[
\frac{p}{x_{1}\cdots x_{n-1}\max(  1,\left\vert k_{n}\right\vert )
\cdots\max(  1,\left\vert k_{d}\right\vert )  }<1,
\]
we obtain (\ref{log^d-n}).
\end{proof}
At last, let us prove Theorem \ref{teo_square}.
\begin{proof}[Proof of Theorem \ref{teo_square}]
By Lemma \ref{conto_importante} and Proposition \ref{prop7},%
\begin{align*}
&\|D_N\|^2_2
 = \frac{1}{M^{2d}(M/2-1)}\sum_{0\neq k
\in\mathbb{J}_M^d}\left|  \sum_{n=1}%
^{N}\exp(  -2\pi i k\cdot M^{-1}z_{n})  \right|
^{2}\sum_{s\in\mathbb{S}_M}\left|  \widehat{\chi}_{-Q(s)  }(
k)  \right|  ^{2}\\
& \geqslant \frac{4^{d+1}\eta_{d}(
\varepsilon) }{M^{2d+1}\pi^{2d}}\left\lfloor \frac{M}{4}\right\rfloor^{2d+1}\sum_{\substack{k\neq0,\\ k\in \mathbb J_M^d\cap [-\varepsilon M,\varepsilon M]^d}}\left|  \sum_{n=1}%
^{N}\exp(  -2\pi i k\cdot M^{-1}z_{n})  \right|
^{2} \frac{1 }{\big(  \prod_{u=1}^{d}%
\max(  1, |k_{u}|)  \big)  ^{2}}.
\end{align*}
By Proposition \ref{int} and Lemma \ref{log},   
\begin{align*}
&  \frac{1 }{  \big(\prod_{u=1}^{d}%
\max(  1, |k_{u}|) \big)^2   }=\frac{1}{p^2}\bigg(\frac{p }{  \prod_{u=1}^{d}%
\max(  1, |k_{u}|)    }\bigg)^2\\
&\quad\geqslant \frac{1}{p^2}\left(\frac{2e}{d-1}\right)^{d-1}\bigg(\log_+\frac{p}{  \prod_{u=1}^{d}%
\max(  1, |k_{u}|    )}\bigg)^{d-1}\\
&\quad=\frac{(d-1)!}{p^2}\left(\frac{2e}{d-1}\right)^{d-1}\int_{1}^{p}\cdots\int_{1}^{\frac{p}{x_{1}\cdots x_{d-2}}}%
\chi_{R_{\left(  x_{1},\ldots,x_{d-1}\right)  }}(k)\frac{dx_{d-1}%
}{x_{d-1}}\cdots\dfrac{dx_{1}}{x_{1}}.
\end{align*}
Therefore, if $M\geqslant 4$ so that $\lfloor M/4 \rfloor\geqslant M/8$ ,
\begin{align*}
\|D_N\|_2^2&\geqslant\frac{4^{d+1}\eta_{d}(
\varepsilon) }{M^{2d+1}\pi^{2d}}\left( \frac{M}{8}\right)^{2d+1}
\sum_{\substack{k\neq0,\\ k\in \mathbb J_M^d\cap [-\varepsilon M,\varepsilon M]^d}}\left\vert
\sum_{n=1}^{N}\exp(  -2\pi ik\cdot M^{-1}z_{n})
\right\vert ^{2}\\
&\qquad\times\frac{(  d-1)  !}{p^{2}}\left(  \frac{2e}{d-1}%
\right)  ^{  d-1  }\int_{1}^{p}\cdots\int_{1}^{\frac{p}{x_{1}\cdots x_{d-2}}}%
\chi_{R_{\left(  x_{1},\ldots,x_{d-1}\right)  }}(k)\frac{dx_{d-1}%
}{x_{d-1}}\cdots\dfrac{dx_{1}}{x_{1}}\\
& \geqslant\frac{\eta_{d}(
\varepsilon)(d-1)! }{2^{3d+2}\pi^{2d}p^2} \left(  \frac
{e}{d-1}\right)  ^{ d-1  }
\int_{1}^{p}\cdots\int_{1}^{\frac{p}{x_{1}\cdots x_{d-2}}}\sum_{\substack{k\neq0,\\ k\in \mathbb J_M^d\cap [-\varepsilon M,\varepsilon M]^d}} \chi_{R_{(  x_{1},\ldots,x_{d-1})  }}(k)\\
&\qquad\times \left\vert
\sum_{n=1}^{N}\exp(  -2\pi ik\cdot M^{-1}z_{n})
\right\vert ^{2}
\frac{dx_{d-1}%
}{x_{d-1}}\cdots\dfrac{dx_{1}}{x_{1}}.
\end{align*}
If $p\leqslant \varepsilon M$, then, in our range of $x_1,\ldots, x_{d-1}$, the box $R_{\left(  x_{1},\ldots,x_{d-1}\right)}$ is contained in $[-\varepsilon M, \varepsilon M]^d$ and therefore, by the classical Cassels-Montgomery inequality (see \cite{mont,BGG, TravagliniBook}),
\begin{align*}
&\sum_{\substack{k\neq0,\\ k\in \mathbb J_M^d\cap [-\varepsilon M,\varepsilon M]^d}} \chi_{R_{\left(  x_{1},\ldots,x_{d-1}\right)  }}(k)\left\vert
\sum_{n=1}^{N}\exp(  -2\pi ik\cdot M^{-1}z_{n})
\right\vert ^{2}\\
&\qquad \geqslant Nx_1 x_2\cdots x_{d-1}\frac{p}{x_{1}x_2\cdots x_{d-1}}-N^2=pN-N^2.
\end{align*}
Hence, 
\begin{align*}
\|D_N\|^2_2 & 
\geqslant\frac{\eta_{d}(
\varepsilon)(d-1)! }{2^{3d+2}\pi^{2d}p^2} \left(  \frac
{e}{d-1}\right)  ^{ d-1  }
\int_{1}^{p}\cdots\int_{1}^{\frac{p}{x_{1}\cdots x_{d-2}}} (pN-N^2)
\frac{dx_{d-1}%
}{x_{d-1}}\cdots\dfrac{dx_{1}}{x_{1}}\\
& =\frac{\eta_{d}(
\varepsilon)(d-1)! }{2^{3d+2}\pi^{2d}p^2}\left(  \frac{e}{d-1}\right)  ^{d-1}(
pN-N^{2}) \\
&\quad \times\int_{1}^{p}\cdots\int_{1}^{\frac{p}{x_{1}\cdots x_{d-3}}%
}\log\Big(  \frac{p}{x_{1}\cdots x_{d-2}}\Big)  \frac{dx_{d-2}}{x_{d-2}%
}\cdots\dfrac{dx_{1}}{x_{1}}\\
& =\frac{\eta_{d}(
\varepsilon) }{2^{3d+2}\pi^{2d}p^2}\left(
\frac{e}{d-1}\right)  ^{d-1}(  pN-N^{2})  \log(  p)
^{d-1},
\end{align*}
where the last equality follows applying repeatedly \eqref{log^d-n}.

We now set $p=2N\leqslant \varepsilon M$. Then%
\[
\|D_N\|_2^2\geqslant \frac{\eta_{d}(
\varepsilon) }{2^{3d+4}\pi^{2d}}\left(
\frac{e}{d-1}\right)  ^{d-1}\log(  2N)
^{d-1}.
\] 
Setting $\varepsilon=\frac{1}{9d}$, gives the constraint 
$
M\geqslant18d\,N$.
\end{proof}

\section{Proof of the lower bound for balls}
In this last section we prove the bound for the discrepancy for balls. Unlike the case for the squares, we need a little bit of extra care when dealing with points $p_n$'s that are not on the lattice. Similarly to what we did for Theorem \ref{teo_square}, we prove a pair of preliminary lemmas before actually proving Theorem \ref{teo_balls}. To simplify the notation for the rest of the section we set $D_N(x,r)=\mathscr D(x+B_r,\mathcal P_n)$.

\begin{lemma}\label{lem:plancherel_ball}
Given any finite sequence $\mathcal P_N=\{p_n\}_{n=1}^N$ there exists a unique decomposition $p_n=\widetilde p_n+q_n$ where $\widetilde{p}_n\in \mathbb{T}^d\cap M^{-1}\mathbb{Z}^d$ and $q_n\in \left[-\frac{1}{2M},\frac{1}{2M}\right)^d$. Then, for every $0<r< 1/4$ and for every $k\in\mathbb J_M^d$, the discrete Fourier transform of the function $D_N(\cdot, r)$ is given by
\begin{equation}\label{eq:fourier_ball}
\widehat D_N(k,r)= \sum_{n=1}^N \exp(-2\pi ik\cdot\widetilde p_n)\widehat \chi_{-B_r+q_n}(k)-|B_r|NM^d \delta_0(k). 
\end{equation}
It follows that
\begin{equation}\label{eq:plancherel_ball}
\frac{1}{M^{d}}\sum_{j\in\mathbb{J}_M^d}\left\vert D_{N}(M^{-1}j,r)  \right\vert ^{2}   \geqslant\frac{1}{M^{2d}}\sum_{0\neq k\in\mathbb{J}_M^d}\left\vert \sum
_{n=1}^{N}\exp(-2\pi ik\cdot \widetilde{p}_n)\widehat{\chi
}_{-B_{r}+q_n}(k)\right\vert ^{2}.
\end{equation}
\end{lemma}
\begin{proof}
The existence and the uniqueness of the decomposition $p_n=\widetilde p_n+q_n$ is trivial. Then, by definition, 
\begin{align*}
\widehat{D}_N(  k,r)   &  =\sum_{m \in \mathbb J_M^d}D_N(M^{-1}m,r)  \exp(  -2\pi ik %
\cdot M^{-1}m) \\
&=\sum_{m \in \mathbb J_M^d}\bigg(\sum_{n=1}^{N}\chi_{M^{-1}m+B_r  }(  p_{n})  -N|B_r| \bigg)  \exp(  -2\pi ik %
\cdot M^{-1}m)\\
&=\sum_{m \in \mathbb J_M^d}\bigg(\sum_{n=1}^{N}\chi_{M^{-1}m+B_r -q_n }(  \widetilde{p}_{n})  -N|B_r| \bigg)  \exp(  -2\pi ik %
\cdot M^{-1}m)\\
&  =\sum_{n=1}^{N}\sum_{m \in \mathbb J_M^d}\chi_{-B_r+q_n  }(M^{-1}(m-M\widetilde{p}_n))  \exp(  -2\pi ik\cdot
M^{-1}m)  -|B_r|NM^{d}\delta_0(
k)  \\
&= \sum_{n=1}^{N}\exp(-2\pi ik\cdot\widetilde{p}_n)\sum_{m \in \mathbb J_M^d}\chi_{-B_r+q_n  }(M^{-1}m)  \exp(  -2\pi ik\cdot
M^{-1}m)  -|B_r|NM^{d}\delta_0(
k) \\
&= \sum_{n=1}^{N}\exp(-2\pi ik\cdot\widetilde{p}_n)\widehat{\chi}_{-B_r+q_n}(k)  -|B_r|NM^{d}\delta_0(
k),
\end{align*}
and \eqref{eq:fourier_ball} is proved. To prove \eqref{eq:plancherel_ball} simply observe that, by Plancherel formula \eqref{Plancherel},
\begin{align*}
\frac{1}{M^{d}}\sum_{j\in\mathbb{J}_M^d}\left\vert D_{N}(M^{-1}j,r)  \right\vert ^{2}  &
=\frac{1}{M^{2d}}\sum_{k\in\mathbb{J}_M^d}\vert \widehat{D}_N(k,r) \vert ^{2}\\
&  \geqslant\frac{1}{M^{2d}}\sum_{0\neq k\in\mathbb{J}_M^d}\left\vert \sum
_{n=1}^{N}\exp(-2\pi ik\cdot \widetilde{p}_n)\widehat{\chi
}_{-B_{r}+q_n}(k)\right\vert ^{2}.
\end{align*}
\end{proof}
We now want to further investigate $\widehat \chi_{-B_r+q_n}(k)$. Let us introduce some notation. Let \[
Q_0=\left[-\frac12,\frac12\right)^d\subseteq\mathbb R^d
\] 
and for $\xi\in\mathbb R^d$ set
\[
G(  \xi)  =\int_{\mathbb{R}^d}\chi_{Q_{0}}(x)\exp(-2\pi i\xi\cdot x)dx=\prod_{j=1}^{d}%
\frac{\sin\pi\xi_{j}}{\pi\xi_{j}}.
\]
For $0<r<1/4$ and $q\in \left[-\frac{1}{2M},\frac 1{2M}\right)^d$ set
\[
A_{r,M,q}=\bigcup_{m\in M(-B_r+q)}
(  m+Q_{0})  .
\]
and
\[
E_{r,M,q}=\Big\{x:  \left\vert x\right\vert \leqslant rM+\sqrt
{d}\Big\}  \setminus A_{r,M,q}.
\]
Then, using the above notation, we have the following.
\begin{lemma}\label{lem:fourier_ball}
Let $J_{d/2}$ be the Bessel function of order $d/2$. For every $q\in \left[-\frac{1}{2M},\frac{1}{2M}\right)^d$, for every $0<r< 1/4$ and for every $0\neq k\in\mathbb J_M^d$,
\begin{align*}
\widehat \chi_{-B_r+q}(k)&=\frac{1}{G(  M^{-1}k)  }(  rM+\sqrt{d})
^{\frac{d}{2}}\frac{J_{d/2}\bigg(  2\pi\dfrac{rM+\sqrt{d}}{M}
\left\vert k\right\vert \bigg)  }{(  \left\vert k\right\vert
/M)  ^{\frac{d}{2}}}\\ 
&\quad-\frac{1}{G( M^{-1} k)  }%
\int_{E_{r,M,q}}\exp(2\pi iM^{-1}k\cdot x)dx,
\end{align*}
\end{lemma}
\begin{proof}
By definition,
\begin{align*}
\widehat{\chi}_{-B_{r}+q}(k)  = & \sum_{m\in\mathbb{J}_M^d}\chi_{-B_r+q}(M^{-1}m)\exp(-2\pi ik\cdot M^{-1}m)
= \sum_{m\in M(-B_r+q)}\exp(-2\pi ik\cdot M^{-1}m).
\end{align*}
Since
\[
G(  \xi)  =\int_{\mathbb{R}^d}\chi_{Q_{0}}(x)\exp(-2\pi i\xi\cdot x)dx,
\]
then
\[
\int_{\mathbb{R}^d}\chi_{m+Q_{0}}(x)\exp(-2\pi i\xi\cdot x)dx=\exp(2\pi i\xi\cdot m)G(  \xi)  .
\]
It follows that%
\begin{align*}
\widehat{\chi}_{-B_{r}+q}(k)   &  =\sum_{m\in M(-B_r+q)}\exp(-2\pi ik\cdot M^{-1}m)\\
&=\frac{1}{G(  M^{-1}k)  }%
\sum_{m\in M(-B_r+q)}\int_{\mathbb{R}^d}\chi_{m+Q_{0}}(x)\exp(2\pi iM^{-1}k\cdot
x)dx\\
&  =\frac{1}{G(  M^{-1}k)  }\int_{\mathbb{R}^d}\chi_{A_{r,M,q}}(x)\exp(2\pi iM^{-1}k\cdot x)dx.
\end{align*}
Since $A_{r,M,q}\subseteq\Big\{x:  \left\vert x\right\vert \leqslant rM+\sqrt{d}\Big\}  $, we have, by the classical formula for the Fourier transform of the ball (see, e.g, \cite{SW}),
\begin{align*}
\widehat{\chi}_{-B_{r}+q}(k)   &  =\frac{1}{G(  M^{-1}k)
}\int_{\{  \left\vert x\right\vert \leqslant rM+\sqrt
{d}\}  }\exp(2\pi iM^{-1}k\cdot x)dx\\&\quad -\frac{1}{G( M^{-1} )  }%
\int_{E_{r,M,q}}\exp(2\pi iM^{-1}k\cdot x)dx\\
&  =\frac{1}{G(  M^{-1}k)  }(  rM+\sqrt{d})
^{\frac{d}{2}}\frac{J_{d/2}\bigg(  2\pi\dfrac{rM+\sqrt{d}}{M}
\left\vert k\right\vert\bigg)  }{(  \left\vert k\right\vert
/M)  ^{\frac{d}{2}}}\\
& \quad  -\frac{1}{G( M^{-1} k)  }%
\int_{E_{r,M,q}}\exp(2\pi iM^{-1}k\cdot x)dx.
\end{align*}
\end{proof}
We are ready to prove Theorem \ref{teo_balls}. 
\begin{proof}[Proof of Theorem \ref{teo_balls}]
Using the notation of Lemma \ref{lem:fourier_ball}, set 
\[
I_{k,M}(r)=(  rM+\sqrt{d})
^{\frac{d}{2}}\frac{J_{d/2}\bigg(  2\pi\dfrac{rM+\sqrt{d}}{M}
\left\vert k\right\vert\bigg)  }{(  \left\vert k\right\vert
/M)  ^{\frac{d}{2}}}
\]
and
\[R_{k,M,n}(r)=-\int_{E_{r,M,q_n}}\exp(2\pi iM^{-1}k\cdot x)dx.
\]
By Lemma \ref{lem:plancherel_ball} and Lemma \ref{lem:fourier_ball}, since $\left|a+b\right|^2 \geqslant \frac 12 \left|a\right|^2 - \left|b\right|^2 $ and $|G(M^{-1}k)|\leqslant 1$, for any given $1<L<M/2$,
\begin{align*}
&\frac{1}{M^{d}}\sum_{j\in\mathbb{J}_M^d}\left\vert D_{N}(M^{-1}j,r)  \right\vert ^{2}\\ 
&  \geqslant\frac{1}{M^{2d}}\sum_{0<|k|\leqslant L}\Big(\frac{1}{G( M^{-1} k)  }\Big)^2\left\vert \sum
_{n=1}^{N}\exp(-2\pi ik\cdot \widetilde{p}_n)I_{k,M}(r)+\sum
_{n=1}^{N}\exp(-2\pi ik\cdot \widetilde{p}_n)R_{k,M,n}(r)\right\vert ^{2}\\
&\geqslant \frac{1}{2M^{2d}}\sum_{0<|k|\leqslant L}\left|I_{k,M}(r)\right|^2\left\vert \sum
_{n=1}^{N}\exp(-2\pi ik\cdot \widetilde{p}_n)\right\vert^2-\frac{1}{M^{2d}}\sum_{0<|k|\leqslant L}\left\vert\sum 
_{n=1}^{N}\exp(-2\pi ik\cdot \widetilde{p}_n)R_{k,M,n}(r)\right\vert ^{2}\\
& \geqslant \frac{1}{2M^{2d}}\sum_{0<|k|\leqslant L}|I_{k,M}(r)|^2\left\vert \sum
_{n=1}^{N}\exp(-2\pi ik\cdot \widetilde{p}_n)\right\vert^2-c\frac{L^d}{M^{2d}}N^2(rM)^{2d-2},
\end{align*}
where for the last estimate we used that
\[
|R_{k,M,n}(r)|\leqslant |E_{r,M,q_n}|\leqslant c (rM)^{d-1}
\]
uniformly in $n=1,\ldots, N$.
Hence,
\begin{align*}
&\frac{1}{M^{d}}\sum_{j\in\mathbb{J}_M^d}\big(\left\vert D_{N}(M^{-1}j,r)  \right\vert ^{2}+\left\vert D_{N}(M^{-1}j,2r)  \right\vert ^{2}\big)\\
& \geqslant \frac{1}{2M^{2d}}\sum_{0<|k|\leqslant L}\big(|I_{k,M}(r)|^2+|I_{k,M}(2r)|^2\big)\left\vert \sum
_{n=1}^{N}\exp(-2\pi ik\cdot \widetilde{p}_n)\right\vert^2-c\frac{L^d}{M^{2}}N^2r^{2d-2}.
\end{align*}
Now,
\begin{align*}
&  |I_{k,M}(r)|^2+|I_{k,M}(2r)|^2\\
& \geqslant  \left\vert (  rM+\sqrt{d})  ^{\frac{d}{2}}%
\frac{J_{d/2}\left(  2\pi\dfrac{rM+\sqrt{d}}{M}\left\vert
k\right\vert \right)  }{(  \left\vert k\right\vert /M)  ^{\frac{d}{2}}%
}\right\vert ^{2} \quad +\left\vert (  2rM+\sqrt{d})  ^{\frac{d}{2}}%
\frac{J_{d/2}\left(  2\pi\dfrac{2rM+\sqrt{d}}{M}\left\vert
k\right\vert \right)  }{(  \left\vert k\right\vert /M)  ^{\frac{d}{2}}%
}\right\vert ^{2}\\
&  \geqslant\frac{M^{2d}r^{d}}{\left\vert k\right\vert ^{d}}\left[
\left\vert J_{d/2}\bigg(  2\pi\frac{rM+\sqrt{d}}{M}\left\vert
k\right\vert \bigg)  \right\vert ^{2}+\left\vert J_{d/2}\bigg(  2\pi
\frac{2rM+\sqrt{d}}{M}\left\vert k\right\vert \bigg)  \right\vert
^{2}\right]  
\end{align*}
By the asymptotic expansion of Bessel functions we have%
\[
J_{a}(  w)  =\sqrt{\frac{2}{\pi w}}\cos\Big(  w-a\frac{\pi}%
{2}-\frac{\pi}{4}\Big)  +O(  w^{-\frac{3}{2}})  .
\]
Hence,
\[
\left\vert J_{a}(  w)  \right\vert ^{2}=\frac{2}{\pi w}\cos
^{2}\Big(  w-a\frac{\pi}{2}-\frac{\pi}{4}\Big)  +O(  w^{-2})  .
\]
It follows that%
\begin{align*}
&  \left\vert J_{d/2}\Big(  2\pi\frac{rM+\sqrt{d}}{M}\left\vert
k\right\vert \Big)  \right\vert ^{2}+\left\vert J_{d/2}\Big(  2\pi
\frac{2rM+\sqrt{d}}{M}\left\vert k\right\vert \Big)  \right\vert
^{2}\\
&  =\frac{M}{\pi^{2}(  rM+\sqrt{d})  \left\vert
k\right\vert }\cos^{2}\Big(  2\pi\frac{rM+\sqrt{d}}{M}\left\vert
k\right\vert -(  d+1)  \frac{\pi}{4}\Big)  \\
&\quad   +\frac{M}{\pi^{2}(  2rM+\sqrt{d})  \left\vert
k\right\vert }\cos^{2}\Big(  2\pi\frac{2rM+\sqrt{d}}{M}\left\vert
k\right\vert -(  d+1)  \frac{\pi}{4}\Big)  -O(  \left\vert
rk\right\vert ^{-2})  \\
&  \geqslant\frac{c}{r\left\vert k\right\vert }\left[  \cos^{2}\bigg(  \Big(
2\pi r+\frac{2\pi}{M}\sqrt{d}\Big)  \left\vert k\right\vert -(
d+1)  \frac{\pi}{4}\bigg)  +\cos^{2}\bigg(  \Big(  4\pi r+\frac{2\pi
}{M}\sqrt{d}\Big)  \left\vert k\right\vert -(  d+1)  \frac{\pi
}{4}\bigg)  \right]  \\
&  \quad -O(  \left\vert rk\right\vert ^{-2})
\end{align*}
since $rM\geqslant C$ as a consequence of (\ref{M da sotto}).
Let
\[
\omega_{1}=\Big(  2\pi r+\frac{2\pi}{M}\sqrt{d}\Big)  \left\vert
k\right\vert -(  d+1)  \frac{\pi}{4}%
\]
and%
\[
\omega_{2}=\Big(  4\pi r+\frac{2\pi}{M}\sqrt{d}\Big)  \left\vert
k\right\vert -(  d+1)  \frac{\pi}{4}.
\]
We claim that
\begin{equation}
\cos^{2}(  \omega_{1})  +\cos^{2}(  \omega_{2})
\label{somma coseni}%
\end{equation}
is bounded away from zero for $\left\vert k\right\vert <\frac
{1}{10\sqrt{d}}M$. Indeed, suppose that we have both%
\begin{equation}\label{omega_1}
\left\vert \omega_{1}-\Big(  \frac{\pi}{2}+\ell\pi\Big)  \right\vert
\leqslant\varepsilon
\end{equation}
and%
\begin{equation}\label{omega_2}
\left\vert \omega_{2}-\Big(  \frac{\pi}{2}+h\pi\Big)  \right\vert
\leqslant\varepsilon
\end{equation}
for suitable integers $\ell$ and $h$. Since%
\[
\omega_{2}=2\omega_{1}-\frac{2\pi}{M}\sqrt{d}\left\vert k\right\vert +(
d+1)  \frac{\pi}{4},%
\]
we have, by \eqref{omega_1} and \eqref{omega_2},
\begin{align*}
\varepsilon & \geqslant\left\vert 2\omega_{1}-\frac{2\pi}{M}\sqrt{d}\left\vert
k\right\vert +(  d+1)  \frac{\pi}{4}-\left(  \frac{\pi}{2}%
+h\pi\right)  \right\vert \\
& =\left\vert 2\left(  \omega_{1}-\frac{\pi}{2}-\ell\pi\right)  +2\ell\pi
-h\pi-\frac{2\pi}{M}\sqrt{d}\left\vert k\right\vert +\left(  d+3\right)
\frac{\pi}{4}\right\vert \\
& \geqslant\left\vert 2\ell\pi-h\pi+\left(  d+3\right)  \frac{\pi}%
{4}\right\vert -2\varepsilon-\frac{2\pi}{M}\sqrt{d}\left\vert k\right\vert \\
& =\frac{\pi}{4}\left\vert 8\ell-4h+d+3\right\vert -2\varepsilon-\frac{2\pi}%
{M}\sqrt{d}\left\vert k\right\vert .
\end{align*}

Since $d\not \equiv 1\operatorname{mod}4$, we have $\left\vert 8\ell
-4h+d+3\right\vert \geqslant1$.\ Recalling that $\left\vert k\right\vert <\frac
{1}{10\sqrt{d}}M$, we have%
\[
\varepsilon\geqslant\frac{\pi}{4}-2\varepsilon-\frac{2\pi}{M}\sqrt{d}\left\vert
k\right\vert >\frac{\pi}{4}-2\varepsilon-\frac{\pi}{5}>\frac{\pi}%
{20}-2\varepsilon
\]
which is not possible if $\varepsilon$ small enough. Hence 
(\ref{somma coseni}) is bounded away from zero and therefore
\[
\left\vert J_{d/2}\Big(  2\pi\frac{rM+\sqrt{d}}{M}\left\vert
k\right\vert \Big)  \right\vert ^{2}+\left\vert J_{d/2}\Big(  2\pi
\frac{2rM+\sqrt{d}}{M}\left\vert k\right\vert \Big)  \right\vert
^{2}\geqslant\frac{c_{1}}{r\left\vert k\right\vert }-\frac{c_{2}}{(
r\left\vert k\right\vert )  ^{2}}.
\]
It follows that%
\begin{align*}
&  |I_{k,M}(r)|^2+|I_{k,M}(2r)|^2\\
&  \geqslant\frac{M^{2d}r^{d}}{\left\vert k\right\vert ^{d}}\left[
\left\vert J_{d/2}\Big(  2\pi\frac{rM+\sqrt{d}}{M}\left\vert
k\right\vert \Big)  \right\vert ^{2}+\left\vert J_{d/2}\Big(  2\pi
\frac{2rM+\sqrt{d}}{M}\left\vert k\right\vert \Big)  \right\vert
^{2}\right] \\
&\geqslant \frac{M^{2d}r^{d}}{\left\vert k\right\vert ^{d}}\Big( \frac{c_{1}}{r\left\vert k\right\vert }-\frac{c_{2}}{\left(
r\left\vert k\right\vert \right)  ^{2}}\Big) \geqslant c_3\frac{M^{2d}r^{d-1}}{\left\vert k\right\vert ^{d+1}},
\end{align*}
for $\left\vert k\right\vert >\frac{2c_{2}}{  c_{1}r}  $. Let $1<L<\frac
{1}{10\sqrt{d}}M$.
Then, by Cassels-Montgomery inequality (see \cite{mont,BGG, GGbis, TravagliniBook}),%
\begin{align*}
&\frac{1}{M^{d}}\sum_{j\in\mathbb{J}_M^d}\left(\left\vert D_{N}(M^{-1}j,r)  \right\vert ^{2}+\left\vert D_{N}(M^{-1}j,2r)  \right\vert ^{2}\right)\\
& \geqslant \frac{1}{2M^{2d}}\sum_{\frac{2c_{2}}{c_{1}r}
<\left\vert k\right\vert \leqslant L}\big(|I_{k,M}(r)|^2+|I_{k,M}(2r)|^2\big)\left\vert \sum
_{n=1}^{N}\exp(-2\pi ik\cdot \widetilde{p}_n)\right\vert^2-c\frac{L^d}{M^{2}}N^2r^{2d-2}\\
& \geqslant \frac{1}{2M^{2d}}\sum_{\frac{2c_{2}}{c_{1}r}
<\left\vert k\right\vert \leqslant L}\Big(c_3\frac{M^{2d}r^{d-1}}{\left\vert k\right\vert ^{d+1}}\Big)\left\vert \sum
_{n=1}^{N}\exp(-2\pi ik\cdot \widetilde{p}_n)\right\vert^2-c\frac{L^d}{M^{2}}N^2r^{2d-2}\\
&  \geqslant\frac{c_{3}}{2L^{d+1}}r^{d-1}\sum_{\frac{2c_{2}}{  c_{1}r}
<\left\vert k\right\vert \leqslant L}\left\vert \sum_{n=1}^{N}\exp(-2\pi ik\cdot
\widetilde{p}_n)\right\vert ^{2}-c\frac{L^d}{M^{2}}N^2r^{2d-2}\\
&  \geqslant c_{4}\frac{r^{d-1}}{L^{d+1}}\big(  NL^{d}-\frac{c_{5}}{r^{d}}%
N^{2}\big)  -c\frac{L^d}{M^{2}}N^2r^{2d-2}.
\end{align*}
Let $L^{d}=2\frac{c_{5}}{r^{d}}N$. Then%
\begin{align*}
 &\frac{1}{M^{d}}\sum_{j\in\mathbb{J}_M^d}\left(\left\vert D_{N}(M^{-1}j,r)  \right\vert ^{2}+\left\vert D_{N}(M^{-1}j,2r)  \right\vert ^{2}\right) \geqslant C_{1}r^dN^{1-\frac{1}{d}}-C_{2}\frac{N^{3}r^{d-2}}{M^{2}%
} \geqslant \frac{C_1}{2}r^dN^{1-\frac{1}{d}}%
\end{align*}
if $C_{1}r^dN^{1-\frac{1}{d}}\geqslant2C_{2}\frac{N^{3}r^{d-2}}{M^{2}}$, that is%
\[
M\geqslant C_{3}N^{1+\frac{1}{2d}  }r^{-1}.
\]
Observe that this last condition on $M$ assures that the value chosen above for $L$ is compatible with the condition $1<L<\frac{1}{10\sqrt{d}}M$.
\end{proof}

\bibliographystyle{amsalpha}
\bibliography{biblio_Roth_Discreto}

\end{document}